\documentclass{article}
\usepackage{amsfonts,amsmath, amssymb,latexsym,mathrsfs}
\usepackage[OT2,OT1]{fontenc}
\usepackage[all,cmtip]{xy}
\def\SH{\mbox{\fontencoding{OT2}\selectfont\char88}}

\def\Z{{\mathbb Z}}

\def\SL{{\rm SL}}
\def\GL{{\rm GL}}
\def\PGL{{\rm PGL}}

\def\Stab{{\rm Stab}}

\def\Jac{{\rm Jac}}

\def\SO{{\rm SO}}
\def\P{{\mathbb P}}
\def\Disc{{\rm Disc}}

\def\Aut{{\rm Aut}}
\def\irr{{\rm irr}}

\def\red{{\rm red}}
\def\Inv{{\rm Inv}}
\def\Vol{{\rm Vol}}
\def\bigstab{{\rm bigstab}}
\def\R{{\mathbb R}}
\def\F{{\mathbb F}}
\def\FF{{\mathcal F}}
\def\RR{{\mathcal R}}

\def\Q{{\mathbb Q}}
\def\pv{{V_\Z^{+}}}
\def\nv{{V_\Z^{-}}}
\def\pnv{{V_\Z^{\pm}}}
\def\pr{{V_\R^{+}}}
\def\nr{{V_\R^{-}}}
\def\pnr{{V_\R^{\pm}}}

\def\pRR{{\mathcal R_X(\pl)}}
\def\nRR{{\mathcal R_X(\nl)}}

\def\pnRR{{\mathcal R_X(\pnl)}}
\def\pnRh{{\mathcal R_X(\pnl,h)}}
\def\pnRhl{{\mathcal R_X(h\cdot L)}}
\def\pnRh{{\mathcal R_X(h\cdot\pnl)}}
\def\pRh{{\mathcal R_X(h\cdot\pl)}}
\def\nRh{{\mathcal R_X(h\cdot\nl)}}

\def\rrpx{{R_V^{+}(X)}}

\def\pl{{L^{+}}}
\def\nl{{L^{-}}}
\def\pnl{{L^{\pm}}}
\def\H{{\mathcal H}}

\def\J{{\mathcal J}}
\def\C{{\mathcal C}}

\def\W{{\mathcal W}}

\def\Z{{\mathbb Z}}
\def\P{{\mathbb P}}
\def\F{{\mathbb F}}
\def\Q{{\mathbb Q}}
\def\C{{\mathbb C}}

\def\H{{\mathcal H}}

\newtheorem{theorem}{Theorem}
\newtheorem{corollary}[theorem]{Corollary}
\newtheorem{lemma}[theorem]{Lemma}
\newtheorem{proposition}[theorem]{Proposition}
\newenvironment{proof}{\noindent {\bf Proof:}}{$\Box$ \vspace{2 ex}}

\title{Ternary cubic forms having bounded invariants, and the existence of a
positive proportion of elliptic curves having rank 0}

\author{Manjul Bhargava and Arul Shankar}

\begin{document}
\maketitle


\section{Introduction}

Any elliptic curve $E$ over $\Q$ is isomorphic to a unique curve of
the form $E_{A,B}:y^2=x^3+Ax+B$, where $A,B \in \Z$ and for all primes
$p$:\, $p^6 \nmid B$ whenever $p^4 \mid A$. The (naive) {\it height}
$H(E_{A,B})$ of the elliptic curve $E=E_{A,B}$ is then defined by
$$H(E_{A,B}):= \max\{4|A^3|,27B^2\}.$$

In a previous paper~\cite{BS}, we showed that 
the average rank of all elliptic curves, when ordered by height, is
finite.  This was accomplished by proving that the average size of the
2-Selmer group of elliptic curves, when ordered by height, is
exactly 3; it then followed from the latter result that (the limsup
of) the average
rank of all elliptic curves is bounded above by 1.5.

In this article, we prove an analogous result for the average size of
the 3-Selmer group: 


\begin{theorem}\label{mainellip}
When all elliptic curves $E/\Q$ are ordered by height,
the average size of the $3$-Selmer group $S_{3}(E)$ is $4$.
\end{theorem}

The above result is also seen to imply
the boundedness of the average rank of all elliptic curves.  Indeed,
for an elliptic curve $E$ over $\Q$, since
the 3-rank $r_3(S_3(E))$ of the $3$-Selmer group $S_3(E)$ of $E$
bounds the rank of $E$, and since 
$6r_3(E)-3\leq 3^{r_3(E)}=|S_3(E)|$, by taking averages we immediately
obtain the 
following improved
bound on the average rank of elliptic curves:


\begin{corollary}\label{corellip}
  When all elliptic curves over $\Q$ are ordered by height, their average
  $3$-Selmer rank is at most $1\frac{1}{6}$; thus their average rank is also
  at most $1\frac{1}{6}<1.17$.
\end{corollary}
Theorem~\ref{mainellip} also yields the same bound of
$1\frac16$ on the average 3-rank of the Tate-Shafarevich group of all
elliptic curves, when ordered by height.

We will in fact prove a stronger version of Theorem \ref{mainellip},
namely:
\begin{theorem}\label{ellipcong}
  When elliptic curves $E:y^2=x^3+Ax+B$, in any family defined by finitely
  many congruence conditions on the coefficients $A$ and $B$, are ordered by
  height, the average size of the $3$-Selmer group $S_3(E)$ is~$4$.
\end{theorem}
Thus the average size of the 3-Selmer group remains 4 even when one
averages over any subset of elliptic curves defined by finitely many 
local conditions. We will actually prove Theorem~\ref{ellipcong} 
for an even larger class of families, including some that are defined by
certain natural {infinite} sets of local conditions (such as
the family of all {\it semistable} elliptic curves). 

Theorem~\ref{ellipcong}, and its above-mentioned extensions, allow us
to deduce a number of additional results on ranks that could not be deduced
solely through understanding the average size of the 2-Selmer group,
as in~\cite{BS}. First, by combining our counting techniques with the
remarkable recent results of Dokchitser--Dokchitser~\cite{DD} on the
parity of $p$-ranks of Selmer groups, we prove:

\begin{theorem}\label{algrank0}
  When all elliptic curves $E/\Q$ are ordered by height, a positive
  proportion of them have rank $0$.
\end{theorem}
In the case of rank 1, if we assume the finiteness of the
Tate-Shafarevich group, then we also have:

\begin{theorem}\label{algrank1}
Assume {\rm \SH(}$E)$ is finite for all $E$.  When all elliptic
curves $E/\Q $ are ordered by height, a positive proportion of them
have rank~$1$.
\end{theorem}

Next, combining our counting arguments with the important recent work
of Skinner--Urban~\cite{SU} on the Iwasawa Main Conjectures for
$\GL_2$, we obtain:

\begin{theorem}\label{anrank0}
  When all elliptic curves $E/\Q$ are ordered by height, a positive
  proportion of them have {\em analytic} rank $0$; that is, a positive
  proportion of elliptic curves have nonvanishing {$L$}-function at
  $s=1$.
\end{theorem}

Applying Kolyvagin's Theorem, or noting that the elliptic curves of
analytic rank 0 that arise in Theorem~\ref{anrank0} form a subset of
those that are constructed in Theorem~\ref{algrank0}, we conclude:

\begin{corollary}\label{bsdcor}
A positive proportion of elliptic curves satisfy 
{\em BSD}.
\end{corollary}

Our previous results on the average size of the 2-Selmer group were
obtained through counting integral binary quartic forms, up to
$\GL_2(\Z)$-equivalence, having bounded invariants.  The connection
with elliptic curves is that the process of 2-descent has a classical
interpretation in terms of rational binary quartic forms; indeed, this
connection was behind the beautiful computations of Birch and
Swinnerton-Dyer in~\cite{BSD}.  The process of 2-descent through the
use of binary quartic forms, as in Cremona's remarkable {\tt mwrank}
program, remains the fastest method in general for computing ranks of
elliptic curves.

In order to prove an analogous result for the average size of
3-Selmer groups, we apply our counting techniques in~\cite{BS},
appropriately modified, to the space $V_\Z$ of {\bf integral ternary
  cubic forms}.  The group $\SL_3(\Z)$ acts naturally on $V_\Z$, and
the ring of polynomial invariants over $\C$ 
for this action turns out to have 
two independent generators, 
having degrees 4 and 6, which we denote by $I$ and $J$ respectively.

These invariants may be constructed as follows.  For a ternary
cubic form $f$, let $\H(f)$ denote the {\it Hessian} of $f$, i.e., the
determinant of the $3\times 3$ matrix of second order partial
derivatives of $f$:
\begin{equation}\label{hessdef}
\H(f(x,y,z)):= \left|\begin{array}{ccc}f_{xx} & f_{xy} & f_{xz}\\f_{xy} &
  f_{yy} & f_{yz}\\f_{xz} & f_{yz} & f_{zz}\end{array}\right|.
 \end{equation}
 Then $\H(f)$ is itself a ternary cubic form  and, moreover, it is an
 $\SL_3$-covariant of $f$, i.e., for $\gamma\in\SL_3$, we have
 $\H(\gamma\cdot f)=\gamma\cdot\H(f)$.  An easy computation gives
\begin{equation}\label{defij}
\H(\H(f))\,=\,12288 \,I(f)^2\cdot f\,+\,512 \,J(f)\cdot\H(f)
\end{equation}
for certain rational polynomials $I(f)$ and $J(f)$ in the
coefficients of $f$, having degrees 4
and 6 respectively; note that (\ref{defij}) 
uniquely determines $J(f)$, and also 
uniquely determines $I(f)$ up to sign.
The sign of $I(f)$ is fixed by the requirement that the {\it
  discriminant} $\Delta(f)$ of a ternary cubic form $f$ be expressible
in terms of $I(f)$ and $J(f)$ by the same formula as for binary
quartic forms, namely
\begin{equation}\label{discdef}
\Delta(f):=\Delta(I,J):=(4I(f)^3-J(f)^2)/27.
\end{equation}
These polynomials $I(f)$ and $J(f)$ are evidently
$\SL_3$-invariants, and in fact they generate the full ring of
polynomial invariants over $\C$.

Traditionally, the generators of the ring of invariants of the action
of $\SL_3$ on the space of ternary cubic forms have been denoted by
$S$ and $T$ (called the Aronhold invariants~\cite{A}), 
which are certain integer
multiples of $I$ and~$J$, respectively; explicitly, we have $S=16\cdot
I$ and $T=32\cdot J$. However, any ternary cubic form $f$ having
complex coefficients and nonzero discriminant is
$\SL_3(\C)$-equivalent to a ternary cubic form $E$ in Weierstrass
form. In previous work (see \cite[\S3]{BS}) we had defined invariants
$I(E)$ and $J(E)$ of such $E$, and the invariants $I(f)$ and $J(f)$ of
ternary cubic forms $f$ have been chosen to agree with those same
invariants of $E$.

Now, for ternary cubic forms over the integers, the general work of Borel
and Harish-Chandra~\cite{BH} implies that the number of equivalence
classes of integral ternary cubic forms, having any given fixed values
for these basic invariants $I$ and $J$ (so long as $I$ and $J$ are not
both equal to zero), is finite.  The question thus arises: how many
$\SL_3(\Z)$-classes of integral ternary cubic forms are there, on
average, having invariants $I,J$, as the pair $(I,J)$ varies?

To answer this question, we require a couple of definitions.  Let us
define the {\it height} of a ternary cubic form $f(x,y,z)$ by
\begin{equation}\label{eqhttc}
H(f):=H(I,J):=\max\{|I^3|,J^2/4\}.
\end{equation}
(As usual, the constant
factor $1/4$ on $J^2$ is present for convenience and is not of any
real importance.)  Thus $H(f)$ is a ``degree 12'' function in the
coefficients of $f$, in the sense that $H(\lambda f)=\lambda^{12}H(f)$ 
for any constant~$\lambda$.  We may then order all $\SL_3(\Z)$-classes of
ternary cubic forms $f$ by their height $H(f)$, and we may similarly order all
pairs $(I,J)$ of invariants by their height $H(I,J)$.

As with binary quartic forms, we wish to restrict ourselves to
counting ternary cubic forms that are irreducible in an appropriate
sense.  Being simply {\it irreducible}---i.e., not having a smaller
degree factor---is more a geometric condition rather than an
arithmetic one.  We wish to have a condition that implies that the
ternary cubic form is sufficiently ``generic'' over $\Q$.  The most
convenient notion (also for the applications) turns out to be what we
call strong irreducibility.

Let us say that an integral ternary cubic form $f$ is {\it
  strongly irreducible} if $f$ is irreducible and 
the common zero set of $f$ and its Hessian
$\H(f)$ in $\P^2$ (i.e., the set of {\it flexes} of $f$ in $\P^2$)
contains no rational points.  
We prove:

\begin{theorem}\label{bqcount}
  Let $h(I,J)$ denote the number of
  $\SL_3(\Z)$-equivalence classes of strongly irreducible ternary
  cubic forms having invariants  equal to $I$ and $J$. Then:
  
  \vspace{.02in}
\begin{itemize}
\item[{\rm (a)}] $\displaystyle{\sum_{\substack{\Delta(I,J)>0\\[.01in]H(I,J)<X}}h(I,J)
\,=\,
\frac{\,32\,}{45}\,\zeta(2)\zeta(3) X^{5/6}\,+\,o(X^{5/6});}$
\item[{\rm (b)}] $\displaystyle{\sum_{\substack{\Delta(I,J)<0\\[.01in]H(I,J)<X}}h(I,J)
\,=\;\!
\frac{128}{45}\;\!\zeta(2)\zeta(3) X^{5/6}\,+\,o(X^{5/6}).}$
\end{itemize}
\end{theorem}

In order to obtain the average size of
$h(I,J)$, as $(I,J)$ varies, we first need to know which pairs $(I,J)$
can actually occur as the invariants of an integral ternary cubic
form.  For example, in the case of binary quadratic and binary cubic
forms, the answer is well-known: there is only one invariant---the
{\it discriminant}---and a number occurs as the discriminant of a
binary quadratic (resp.\ cubic) form if and only if it is congruent to
$0$ or $1$~(mod 4).

In the binary quartic case, we proved in~\cite{BS} that a similar
scenario occurs; namely, an $(I,J)\in\Z\times\Z$ is {\it
  eligible}---i.e., it occurs as the invariants of some integer binary
quartic form---if and only if it satisfies any one of a certain
specified finite set of congruence conditions modulo 27 (see
\cite[Theorem~1.7]{BS}).

It turns out that the invariants $(I,J)$ that can occur (i.e., are
{\it eligible}) for an integral
ternary cubic must also satisfy these same conditions modulo 27.
However, there is also now a strictly larger set of possibilities at the prime 2.
Indeed, the pairs
$(I,J)$ that occur for ternary cubic forms need
not even be integral, but 
rather lie in $\frac1{16}\Z\times\frac1{32}\Z$;
and the pairs $(I,J)$ in this set that actually occur are then defined
by certain congruence conditions modulo 64 on $16I$ and $32J$, in
addition to the same congruence conditions modulo~$27$ on $I$ and $J$
that occur for binary quartic forms.

In particular, the set of
integral pairs $(I,J)\in\Z\times\Z$ that occur as invariants for integral ternary
cubic forms is the same as the set of all pairs $(I,J)$ that occur for integral 
binary quartic forms!  
We prove:

\begin{theorem}\label{eligible}
  A pair $(I,J)$ occurs as the pair
  of invariants of an integral ternary cubic form if and only if 
  $(I,J)\in\frac1{16}\Z\times\frac1{32}\Z$, the pair
  $(16I,32J)$ satisfies one of the following congruence conditions
  modulo~$64$:
\begin{itemize}

\item[{\rm (a)}] $16I \equiv \,\,0 \pmod{16}$ and \,$32J \equiv \,0 \pmod{32},$
{\rm \,\,(f)\,} $16I\equiv 25 \pmod{64}$ and\, $32J \equiv \;3 \;\pmod{32},$

\item[{\rm (b)}] $16I\equiv \,\,0 \pmod{16}$ and \,$32J \equiv \,8 \pmod{32},$
{\rm \,\,(g)\,} $16I \equiv 33 \pmod{64}$ and\, $32J \equiv 15 \pmod{32},$

\item[{\rm (c)}] $16I \equiv \,\,1 \pmod{64}$ and \,$32J \equiv \!31 \!\pmod{32},$
{\rm \,\,(h)\!\,} $16I\equiv 41 \pmod{64}$ and\, $32J \equiv 11 \pmod{32},$

\item[{\rm (d)}] $16I\equiv \,\,9 \pmod{64}$ and \,$32J \equiv \!27 \!\pmod{32},$
{\rm \,\,\,(i)\,} $16I \equiv 49 \pmod{64}$ and\, $32J \equiv 23 \pmod{32},$

\item[{\rm (e)}] $16I \equiv 17 \!\pmod{64}$ and \,$32J \equiv \,7 \pmod{32},$
{\rm \,\,\,(j)\,} $16I\equiv 57 \pmod{64}$ and\, $32J \equiv 19 \pmod{32},$
\end{itemize}
and $(I,J)$ satisfies one of the following
  congruence conditions modulo $27$:
\begin{itemize}
\item[{\rm (a)}] $I \equiv 0 \pmod3$ and $J \equiv \;\,\,\,0\;\, \pmod{27},$

\item[{\rm (b)}] $I\equiv 1 \pmod9$ and $J \equiv \;\pm 2\, \pmod{27},$

\item[{\rm (c)}] $I\equiv 4 \pmod9$ and $J \equiv \pm 16 \pmod{27},$

\item[{\rm (d)}] $I\equiv 7 \pmod9$ and $J \equiv \;\pm 7\, \pmod{27}.$
\end{itemize}

\end{theorem}

We note that these additional possibilities for the invariants
$(I,J)\in \frac1{16} \Z\times\frac1{32}\Z$ that arise for ternary
cubic forms also arise in the case of binary quartics, provided one
uses ``generalized binary quartics''; see \cite{CFS} or \cite{Fisher1} for
details on the construction and uses of these generalized quartics.

From Theorem~\ref{eligible}, we then conclude that the number of
eligible pairs $(I,J)\in\frac1{16}\Z\times\frac1{32}\Z$, with
$H(I,J)<X$, is asymptotically a certain constant times $X^{5/6}$. By
Theorem~\ref{bqcount}, the number of classes of strongly irreducible
ternary cubic forms, per eligible $(I,J)\in\frac1{16}
\Z\times\frac1{32}\Z$, is thus a constant on average.  We have:

\begin{theorem}\label{bqaverage}
  Let $h(I,J)$ denote the number of $\SL_3(\Z)$-equivalence
  classes of strongly irreducible integral ternary cubic forms having invariants 
  equal to $I$ and $J$. Then:
  
$${\rm (a)}\;\;\displaystyle\lim_{X\rightarrow\infty}
\displaystyle\frac{\displaystyle\sum_{\substack{\Delta(I,J)>0\\H(I,J)<X}}h(I,J)}
{\displaystyle\sum_{\substack{(I,J) \mbox{ {\rm \scriptsize{eligible}}
     }\\[.025in]\Delta(I,J)>0\\[.025in]H(I,J)<X}}1}\,=\,3\zeta(2)\zeta(3);
     \qquad\quad
     {\rm (b)}\;\;\displaystyle\lim_{X\rightarrow\infty}
\displaystyle\frac{\displaystyle\sum_{\substack{\Delta(I,J)<0\\H(I,J)<X}}h(I,J)}
{\displaystyle\sum_{\substack{(I,J) \mbox{ {\rm \scriptsize{eligible}}
     }\\[.025in]\Delta(I,J)<0\\[.025in]H(I,J)<X}}1}\,=\,3\zeta(2)\zeta(3).
     $$
\end{theorem}

The fact that this {\it class number} $h(I,J)$ is a finite constant on
average is indeed what allows us to show that the size of the 3-Selmer
group of elliptic curves too is a finite constant on average.

We actually prove a strengthening of Theorem~\ref{bqaverage}; namely,
we obtain the asymptotic count
of ternary cubic forms having bounded invariants that satisfy
any specified finite set of congruence conditions (see \S\ref{congse},
Theorem~\ref{cong2}).  This strengthening turns out to be crucial for
the application to 3-Selmer groups (as in Theorem~\ref{mainellip}),
which we now discuss.

\vspace{0.09in} Recall that, for any positive integer $n$, an element of
the $n$-Selmer group $S_n(E)$ of an elliptic curve $E/\Q$ may be
thought of as a {``locally soluble $n$-covering''}.  An {\it
  $n$-covering} of $E/\Q$ is a genus one curve $C$ together with maps
$\phi:C\to E$ and $\theta:C\to E$, where $\phi$ is an isomorphism
defined over $\C$, and $\theta$ is a degree $n^2$ map defined over
$\Q$ such that the following diagram commutes:
$$\xymatrix{E \ar[r]^{[n]} &E\\C\ar[u]^\phi\ar[ur]_\theta}$$
Thus an $n$-covering $C=(C,\phi,\theta)$ may be viewed as a 
``twist over $\Q$ of the multiplication-by-$n$ map on $E$''.   
Two $n$-coverings $C$ and $C'$ are said to be {\it isomorphic} if
there exists an isomorphism $\Phi:C\to C'$ defined over $\Q$, and an
$n$-torsion point $P\in E$, such that the following diagram
commutes:
$$\xymatrix{E \ar[r]^{+P} &E\\C\ar[u]^\phi\ar[r]_\Phi &C'\ar[u]_{\phi'}}$$
A {\it soluble $n$-covering} $C$ is one that possesses a rational
point, while a {\it locally soluble $n$-covering} $C$ is one that
possesses an $\R$-point and a $\Q_p$-point for all primes $p$.  Then
we have the isomorphisms:
\begin{eqnarray*} 
\{\mbox{\rm soluble $n$-coverings}\}/\sim &\,\cong\,& E(\Q)/nE(\Q); \\
\{\mbox{\rm locally soluble $n$-coverings}\}/\sim &\,\cong\,& S_{n}(E).
\end{eqnarray*}

Now, counting elements of $S_{3}(E)$ leads to counting ternary cubic
forms for the following reason.  There is a result of Cassels (see
\cite[Theorem 1.3]{Cassels}) that states that any locally soluble $n$-covering
$C$ posseses a degree $n$ divisor defined over $\Q$.  If $n=3$, we
thus obtain an embedding of $C$ into $\P^2$, thereby yielding a
ternary cubic form, well-defined up to $\GL_3(\Q)$-equivalence!
Conversely, given any ternary cubic form $f$ having rational
coefficients and nonzero discriminant, there exists a 3-covering
defined over $\Q$ from the plane cubic $C$ defined by the
equation $f=0$ to the elliptic curve $\Jac(C)$, where $\Jac(C)$ is the
Jacobian of $C$ and is given by the equation
\begin{equation}\label{jac}
  Y^2=X^3-\frac{I(f)}{3}X-\frac{J(f)}{27};\end{equation}
an explicit formula for this 3-covering map may be given in terms
of the $\SL_3$-covariants of $f$ (see~\cite[\S3.2]{MMM}).  Note that
(\ref{jac}) gives another nice interpretation for the invariants
$I(f)$ and $J(f)$ of a ternary cubic form $f$.

To carry out the proof of Theorems~\ref{mainellip} and
\ref{ellipcong}, we do the following:
\begin{itemize}
\item For each rational elliptic curve $E_{A,B}$ and element
  $\sigma\in S_{3}(E_{A,B})$, choose a weighted finite set $S_\sigma$ of {\it
    integral} ternary cubic forms, such that
\begin{itemize}
\item the sum of the weights of the elements in $S_\sigma$ is equal to one;
\item each element $f\in S_\sigma$ gives the 3-covering $\sigma$;
\item the invariants $(I(f),J(f))$ of each element $f\in S_\sigma$ agree with
  the invariants $(A,B)$ of the elliptic curve;
\item the weighted set $S=\displaystyle\bigcup_{{A,B}}\bigcup_{\sigma}S_\sigma$ is defined by congruence conditions.
\end{itemize}
The work of Cremona, Fisher, and Stoll \cite{CFS} on ``minimization''
for ternary cubic forms plays a key role in this
construction.
\item Count these weighted integral ternary cubic forms via a weighted
  congruence version of
  Theorem~\ref{bqcount}. The relevant weighted set $S$ of ternary cubic forms is
  defined by infinitely many
  congruence conditions, so a suitable sieve has to be performed.
\end{itemize}
In the last step, we first use a simple sieve to obtain the
optimal upper bounds. The optimal lower bounds, on the other hand, are significantly more difficult
to obtain, and we use the techniques and results of \cite{geosieve} in
order to prove them.

We may compare Theorem~\ref{mainellip} with a result of de
Jong~\cite{deJong}, who showed that for a finite field of
characteristic not equal to 3, the average size of the 3-Selmer group
of all elliptic curves over $\F_q(t)$ is at most $4+\varepsilon(q)$,
for an explicit function $\varepsilon(q)$ that tends to 0 as
$q\to\infty$. The technique in \cite{deJong} was also essentially that of counting
ternary cubic forms over $\F_q(t)$!  Our main result,
Theorem~\ref{mainellip}, may thus be viewed as a precise version of de
Jong's Theorem over the number field $\Q$.  For more on the history of
average ranks of elliptic curves in families, and related results, 
see \cite{BMSW} and
\cite[\S1]{BS}.

This paper is organized as follows. In Section 2, following the
methods of \cite{BS}, we determine the asymptotic number of
$\SL_3(\Z)$-equivalence classes of strongly irreducible integral
ternary cubic forms having bounded height; in particular, we prove
Theorems~\ref{bqcount}, \ref{eligible}, and \ref{bqaverage}.  The
primary method is that of reduction theory, allowing us to reduce the
problem to counting integral points in certain finite volume regions
in $\R^{10}$.  However, the difficulty in such a count, as usual, lies
in the fact that these regions are not compact, but rather have cusps
going off to infinity.  By studying the geometry of these regions via
the averaging method of \cite{dodpf}, we are able to isolate the
subregions of the fundamental domains that contain predominantly (and
all of the) strongly
irreducible points.  The appropriate volume computations for these
subregions are then carried out to obtain the desired~result.

In Section 3, we then describe the precise correspondence between
ternary cubic forms and elements of the $3$-Selmer groups of elliptic
curves.  We show, in particular, that nonidentity elements of the
3-Selmer group correspond to {strongly irreducible} ternary cubic
forms.  We then apply this correspondence, together with the counting
results of Section~2 and a simple sieve (which involves the
determination of certain local mass formulae for 3-coverings of
elliptic curves over $\Q_p$), to prove that the average size of the
$3$-Selmer groups of elliptic curves, when ordered by height, is at
most~4.  We then use the methods of \cite{geosieve}
to obtain the same lower bound on the average size of the $3$-Selmer
groups of elliptic curves, thus proving Theorems~\ref{mainellip} and
\ref{ellipcong}.

Finally, in Section 4, we combine the results of Sections 2 and 3, as well as the
aforementioned results of Dokchitser--Dokchitser~\cite{DD} and
Skinner--Urban~\cite{SU}, to obtain Theorems \ref{algrank0},
\ref{algrank1}, and~\ref{anrank0}.

\section{The number of integral ternary cubic forms having bounded invariants}

Let $V_\R$ denote the space of all ternary cubic forms having
coefficients in $\R$. The group $\GL_3(\R)$ acts on $V_\R$ on the left via linear
substitution of variable; namely, if $\gamma\in\GL_3(\R)$ and $f\in V_\R$, then
$$(\gamma\cdot f)(x,y,z)=f((x,y,z)\cdot\gamma).$$

For a ternary cubic form $f\in V_\R$, let $\H(f)$ denote the Hessian
covariant of $f$, defined by (\ref{hessdef}), and let $I(f)$ and
$J(f)$ denote the two fundamental polynomial invariants of $f$ as in
(\ref{defij}).  As noted earlier, these polynomials $I(f)$ and $J(f)$ 
are invariant
under the action of $\SL_3(\R)\subset\GL_3(\R)$
and, moreover, they are {\it relative invariants} of
degrees $4$ and $6$, respectively, for the action of $\GL_3(\R)$ on
$V_\R$; i.e, $I(\gamma\cdot f)=\det(\gamma)^4I(f)$ and $J(\gamma\cdot
f)=\det(\gamma)^6J(f)$ for $\gamma\in\GL_3(\R)$ and $f\in V_\R$.  

The discriminant $\Delta(f)$ of a ternary cubic form $f$ 
is a relative invariant of degree 12 and is given by the formula
$\Delta(f)=\Delta(I,J)=
(4I(f)^3-J(f)^2)/27$.  
We define the {\it height} $H(f)$ of $f$ by
$$H(f):=H(I,J):=\max\{|I(f)|^3,J(f)^2/4\}.$$ 
Note that the height is also a degree $12$ relative invariant 
for the action of $\GL_3(\R)$ on $V_\R$.

The action of $\SL_3(\Z)\subset\GL_3(\R)$ on
$V_\R$ evidently preserves the lattice $V_\Z$ consisting of integral ternary cubic
forms. In fact, it also preserves the two sets $\pv$ and $\nv$ 
consisting of those integral ternary cubics that
have positive and negative discriminant, respectively.

As before, we say that an integral ternary cubic form is {\it strongly
  irreducible} if the corresponding cubic curve in $\P^2$ has no
rational flex. For an $\SL_3(\Z)$-invariant set $S\subset V_\Z$, let
$N(S;X)$ denote the number of $\SL_3(\Z)$-equivalence classes of
strongly irreducible elements in $S$ having height less than $X$.
Our purpose in this section is to prove
the following rephrasing of Theorem~\ref{bqcount}:
 
\begin{theorem}\label{thsec2main}
  We have
\begin{itemize}
\item[{\rm (a)}]$N(\pv;X)=\displaystyle\frac{\,32\,}{45}\,\zeta(2)\zeta(3)X^{5/6}+o(X^{5/6});$
\item[{\rm (b)}]$N(\nv;X)=\displaystyle\frac{128}{45}\zeta(2)\zeta(3)X^{5/6}+o(X^{5/6}).$
\end{itemize}
\end{theorem}

\subsection{Reduction theory}\label{s21}
Let $\pr$ (resp.\ $\nr$) denote the set of elements in $V_\R$ having
positive (resp.\ negative) discriminant.  We first
construct fundamental sets in $\pnr$ for the action of $\GL^+_3(\R)$
on $\pnr$, where $\GL^+_3(\R)$ is the subgroup of all elements in
$\GL_3(\R)$ having positive determinant.

To this end, let $f$ be a ternary cubic form in $V_\R$ having nonzero
discriminant, and let $C$ denote the cubic curve in $\P^2$ defined by
the equation $f(x,y,z)=0$.  The set of flexes of $C$ is given by the
set of common zeroes of $f$ and $\H(f)$ in $\P^2$, and hence the number of
such flexes is 9 by Bezout's Theorem. As both $f$ and $\H(f)$ have real
coefficients, the flex points of $C$ are either real or come in complex
conjugate pairs. Therefore, since the total number of flex points is odd, $C$
possesses at least one real flex point.

This implies, in particular, that any ternary cubic form over $\R$ is
$\SL_3(\R)$-equivalent to one in Weierstrass form, i.e., one in the form
\begin{equation}\label{eqwf}
f(x,y,z)=x^3+Axz^2+Bz^3-y^2z
\end{equation}
for some $A,B\in\R$.  It can be checked that the ternary cubic form
$f$ in (\ref{eqwf}) has invariants $I(f)$ and $J(f)$ equal to $-3A$
and $-27B$, respectively. Thus, since $I$ and $J$ are relative
invariants of degrees $4$ and $6$, respectively, two ternary cubic
forms $f$ and $g$ over $\R$, having nonzero discriminant, are
$\GL_3^+(\R)$-equivalent if and only if there exists a positive
constant $\lambda\in\R$ such that $I(f)=\lambda^4I(g)$ and
$J(f)=\lambda^6J(g)$. It follows that a fundamental set $\pl$
(resp.\ $\nl$) for the action of $\GL_3^+(\R)$ on $\pr$ (resp.\ $\nr$)
may be constructed by choosing one ternary cubic form, having
invariants $I$ and $J$, for each pair $(I,J)\in\R\times\R$ such that
$H(I,J)=1$ and $4I^3-J^2>0$ (resp.\ $4I^3-J^2<0$).  We may thus choose
\begin{eqnarray*}
\pl&=&\Bigl\{x^3-\frac{1}{3}xz^2-\frac{J}{27}z^3-y^2z:-2< J <2\Bigr\}\\
\nl&=&\Bigl\{x^3-\frac{I}{3}xz^2\pm\frac{2}{27}z^3-y^2z:-1\leq I < 1\Bigr\}\cup
\Bigl\{x^3+\frac{1}{3}xz^2-\frac{J}{27}z^3-y^2z:-2<J < 2\Bigr\}.
\end{eqnarray*}

The key fact that we need about these fundamental sets $\pnl$ is
that the coefficients of all the ternary cubic forms in the $\pnl$ are
bounded. Note also that if $G_0\subset\GL_3^+(\R)$ is any fixed
compact subset then, for any $h\in G_0$, the set $h\cdot\pnl$ is also a
fundamental set for the action of $\GL_3^+(\R)$ on~$\pnv$, and the
coefficients of the forms in $h\cdot\pnl$ are bounded independent of
$h\in G_0$.

We also require the following fact whose proof we postpone to \S\ref{s31}:
\begin{lemma}\label{lemstabsize}
  Let $f\in V_\R$ be any ternary cubic form having nonzero
  discriminant. Then the order of the stabilizer in $\GL_3^+(\R)$ $($and hence in $\SL_3(\R))$ of
  $f$ is $3$.
\end{lemma}

Let $\FF$ denote a fundamental domain in $\GL_3^+(\R)$ for the left action of
$\GL_3^+(\Z)=\SL_3(\Z)$ on $\GL_3^+(\R)$ contained 
in a standard Siegel set~\cite[\S2]{BH}. We may take
$\FF=\{nak\lambda:n\in N'(a),\,a\in A',\,k\in K,\,\lambda\in\Lambda\}$, where
\begin{eqnarray*}
  &K&\!=\!\;\;\;{\mbox{subgroup $\SO_3(\R)\subset\GL_3^+(\R)$
of orthogonal transformations; }} 
\\[.05in]
  &A'&\!\subset\!\!\;\;\,\,\{a(s_1,s_2):s_1,s_2>c\},\\
  &&\;\;\;\;\;\;\;\;{\rm where}\;a(s_1,s_2)=\left(\begin{array}{ccc}
{s_1^{-2}s_2^{-1}} & {} & {} \\ {} & {s_1s_2^{-1}} & {} \\ {} & {} & {s_1s_2^2}
\end{array}\right);\\[.025in]
&N'(a)&\!=\!\!\;\;\,\,\{n(u_1,u_2,u_3):(u_1,u_2,u_3)\in\nu(a)\},\;\\
&&\;\;\;\;\;\;\;\;{\rm where}\; n(u_1,u_2,u_3)=\left(\begin{array}{ccc} 
{1} & {} & {} \\ {u_1} & {1} & {} \\ {u_2} & {u_3} & {1}
\end{array}\right);\\[.025in]
&\Lambda\!&\!=\!\!\;\;\,\,\{\lambda:\lambda>0\},\;\\
&&\;\;\;\;\;\;\;\;{\rm where}\;\lambda=\left(\begin{array}{ccc}
{\lambda} & {} & {} \\ {} & {\lambda} & {} \\ {} & {} & {\lambda}
\end{array}\right);
\end{eqnarray*}
here $\nu(a)$ is a measurable subset of $[-1/2,1/2]^3$ dependent only
on $a\in A'$, and $c>0$ is an absolute constant.

For $h\in\GL_3^+(\R)$, we regard $\FF h\cdot\pnl$ as a multiset, where
the multiplicity of a point $f\in V_\R^\pm$ is equal to
$\#\{g\in\FF:f\in gh\cdot\pnl\}$. As in \cite[\S2.1]{BS}, it follows
that for any $h\in \GL_3^+(\R)$ and $f\in
V_\R^\pm$, the $\SL_3(\Z)$-orbit of $f$ is represented $m(f)$ times in
$\FF\cdot h\pnl$, where
\begin{equation}\label{eqmxfirst}
m(f):=\#\Stab_{\SL_3(\R)}(f)/\#\Stab_{\SL_3(\Z)}(f);  
\end{equation}
i.e., the sum of the multiplicity in $\FF\cdot
h\pnl$ of $f'$, over all forms $f'$ that are $\SL_3(\Z)$-equivalent
to $f$, is equal to $m(f)$.

For any given $A\in\SL_3(\Z)$, the set of elements in $V_\R$ fixed
by $A$ has measure $0$. Since $\SL_3(\Z)$ is countable, we see that
the set $\{f\in V_\R:\#\Stab_{\SL_3(\Z)}(f)>1\}$ has measure $0$ as
well. Thus, by Lemma \ref{lemstabsize}, the multiset $\FF h\cdot\pnl$ is essentially the union of $3$ fundamental
domains for the action of $\SL_3(\Z)$ on $\pnr$. 

For $h\in\GL_3^+(\R)$, let $\pRh$ and $\nRh$ denote the multisets defined by
$$\pnRh:=\{f\in\FF h\cdot\pnl:H(f)<X\}.$$ 
We will show (cf.\ Lemma~\ref{lemtemp2}) that the number of elements in
$\RR_X(h\cdot L^\pm)$ having nontrivial stabilizer in $\SL_3(\Z)$ (in
fact, in $\SL_3(\Q)$) is negligible.  
It follows, by \eqref{eqmxfirst}
and Lemma \ref{lemstabsize}, that the
quantity $3N(\pnv;X)$ is equal to the number of strongly irreducible
integral ternary cubic forms contained in $\pnRh$, up to an error of
$o(X^{5/6})$.

Counting strongly irreducible integer points in a single such $\pnRh$
is difficult because the domain $\pnRh$ is unbounded (although we
will show that it has finite volume). As in
\cite{BS}, we simplify the counting by averaging over lots of such
domains, i.e., by averaging over a continuous range of elements $h$
lying in a certain fixed compact subset of $\GL^+_3(\R)$.

\subsection{Averaging and cutting off the cusp}\label{s22}

Let $G_0\subset\GL_3^+(\R)$ be a compact semialgebraic 
$K$-invariant subset that is
the closure of some nonempty open set in $\GL_3^+(\R)$, such that
every element in $G_0$ has determinant greater than $1$. For any
$\SL_3(\Z)$-invariant set $S\subset \pnv$, let $S^\irr$ denote the set
of strongly irreducible elements of $S$.  We pick $dh$ to be a Haar
measure on $\GL_3^+(\R)$, and we normalize $dh$ as follows: if
$h\in\GL_3^+(\R)$ is equal to $h=n(u)a(s_1,s_2)k\lambda$ in its
Iwasawa decomposition, then
$$
dh=s_1^{-6}s_2^{-6}dud^\times sdkd^\times\lambda,
$$
where $dk$ is Haar measure on $K$ normalized so that $K$ has measure $1$.
Then we have
\begin{equation}
N(S;X)=\frac{\int_{h\in G_0}\#\{\pnRhl\cap S^\irr\}dh\;}{C_{G_0}}+o(X^{5/6})
\end{equation}
where $L=\pnl$ and $C_{G_0}=3\int_{h\in G_0}dh$.

For $na(s_1,s_2)\lambda\in\FF$, let us write
$B(n,s_1,s_2,\lambda,X):=\{f\in na(s_1,s_2)\lambda G_0 L:H(f)<X\}$.
We then have the following equality, which follows from an argument
identical to the proof of \cite[Theorem 2.5]{BS}:
\begin{equation}\label{eqavgimp}
  N(S;X)=\frac1{C_{G_0}}\int_{g\in N'(a)A'\Lambda}\#\{S^\irr\cap 
  B(n,s_1,s_2,\lambda,X)\}s_1^{-6}s_2^{-6}dn\, d^\times t\,d^\times \lambda+o(X^{5/6}).
\end{equation}
To simplify the right hand side of (\ref{eqavgimp}), we require 
the following lemma which states that the set
$B(n,s_1,s_2,\lambda,X)$ contains no strongly irreducible integral points 
if $s_1$ or $s_2$ is large enough (i.e., when we
are in the ``cuspidal regions'' of the fundamental domains):
\begin{lemma}\label{lemnoirred}
  Let $C>1$ be a constant that bounds the absolute values of the
  $x^3$-, $x^2y$-, $xy^2$-, and $x^2z$-coefficients of all the forms
  in $G_0\cdot \pnl$. Then the set $B(n,s_1,s_2,\lambda,X)$ contains no
  strongly irreducible integral ternary cubic forms if
  $s_1>C^{1/3}\lambda/c$ or if $s_2>C^{1/3}\lambda/c^2$.
\end{lemma}
\begin{proof}
  It is easy to see that if $s_1>C^{1/3}\lambda/c$, then the absolute
  values of the $x^3$-, $x^2y$-, and $x^2z$-coefficients of any
  ternary cubic form in $B(n,s_1,s_2,\lambda,X)$ are all less than
  $1$. Therefore, in this case any integral ternary cubic form in
  $B(n,s_1,s_2,\lambda,X)$ must have its $x^3$-, $x^2y$-, and
  $x^2z$-coefficients equal to $0$, and such a form has a rational
  flex at $[1:0:0]\in\P^2$ and so is not strongly irreducible.
  
  Similarly, if $s_2>C^{1/3}\lambda/c^2$, then any integral ternary
  cubic form in $B(n,s_1,s_2,\lambda,X)$ has its $x^3$-, $x^2y$-, and
  $xy^2$-coefficients equal to $0$, and such a form too always has a
  flex at $[1:0:0]\in\P^2$, and so is not strongly irreducible.
\end{proof}

Now let $V_\Z^\red$ denote the set of all integral ternary cubic
forms that are not strongly irreducible. Then we have the following
lemma, which states that the number of {reducible} points---i.e.,
points 
in $V_\Z^\red$---that are in the
``main body'' of the fundamental domains is negligible.

\begin{lemma}\label{lemfewred}
Let $\FF'$ denote the set of elements $na(s_1,s_2)\lambda k\in\FF$ that satisfy
$s_1<C^{1/3}\lambda/c$ and $s_2<C^{1/3}\lambda/c^2$. Then
$$\int_{na(s_1,s_2)\lambda k\in\FF'}\#\{V_\Z^\red\cap B(n,s_1,s_2,
\lambda,X)\}s_1^{-6}s_2^{-6}dn\, d^\times t\,d^\times \lambda dk=o(X^{5/6}).$$
\end{lemma}
We defer the proof of Lemma \ref{lemfewred} to Section 2.5.

To estimate the number of integral points in
$B(n,s_1,s_2,\lambda,X)$, we use the following proposition due to
Davenport \cite{Davenport1}.

\begin{proposition}\label{lemdav}
  Let $\mathcal R$ be a bounded, semi-algebraic multiset in $\R^n$
  having maximum multiplicity $m$, and that is defined by at most $k$
  polynomial inequalities each having degree at most $\ell$.  
Then the number of integer
  lattice points $($counted with multiplicity$)$ contained in
 $\mathcal R$ is
\[\Vol(\mathcal R)+ O(\max\{\Vol(\bar{\mathcal R}),1\}),\]
where $\Vol(\bar{\mathcal R})$ denotes the greatest $d$-dimensional 
volume of any projection of $\mathcal R$ onto a coordinate subspace
obtained by equating $n-d$ coordinates to zero, where 
$d$ takes all values from
$1$ to $n-1$.  The implied constant in the second summand depends
only on $n$, $m$, $k$, and $\ell$.
\end{proposition}

Since every element of $G_0$ was assumed to have determinant greater than
$1$, the set $B(n,s_1,s_2,\lambda,X)$ is empty unless
$c_1\leq\lambda\leq X^{1/36}$, where $c_1$ is a constant such that
$1/c_1^3$ bounds the determinants of all the elements in $G_0$ from
above.  By Equation (\ref{eqavgimp}), Lemmas
\ref{lemnoirred} and \ref{lemfewred}, and Proposition~\ref{lemdav}, we
see that $N(\pnv;X)$ equals
\begin{equation}\label{eqtcomp}
  \frac{1}{C_{G_0}}\int_{\substack{na(s_1,s_2)\lambda k\in\FF'\\c_1\leq\lambda\leq X^{1/36}}}\bigl(\Vol(B(n,s_1,s_2,
  \lambda,X))+O(\Vol(\overline{B(n,s_1,s_2,\lambda,X)}))\bigr)s_1^{-6}s_2^{-6}dn\,
  d^\times s\,d^\times \lambda\, dk + o(X^{5/6})
\end{equation}
because $\Vol(\overline{B(n,s_1,s_2,\lambda,X)})\gg1$ when $\lambda\geq c_1$.

It is easily checked that when $na(s_1,s_2)\lambda k\in\FF'$ and
$\lambda\geq c_1$, the projection of $B(n,s_1,s_2,\lambda,X)$ onto any
coordinate in $V_\R$, apart from the $x^3$- and $x^2y$-coefficients,
is bounded below independent of $n$, $s_1$, $s_2$, and $\lambda$. (For
example, the projection of $B(n,s_1,s_2,\lambda,X)$ onto the
$x^2z$-coefficient is bounded below by an absolute constant times
$\lambda^3s_1^{-3}$, which is bounded from below since $\lambda\geq c_1$ and $s_1\ll\lambda$.)
Thus, the integral of the error term in the
integrand of (\ref{eqtcomp}) is computed to be
$$O\Bigl(\int_{\lambda=0}^{X^{1/36}}\int_{s_1,s_2=c}^{\lambda}(\lambda^{27}s_1^6s_2^3+\lambda^{24}s_1^9s_2^6)s_1^{-6}s_2^{-6}d^\times s\,d^\times \lambda\Bigr)=O(X^{3/4}).$$
Meanwhile, the integral of the main term in the integrand of
(\ref{eqtcomp}) is equal to
\begin{eqnarray*}
&&\frac{1}{C_{G_0}}\int_{\substack{na(s_1,s_2)\lambda k\in\FF'\\c_1\leq\lambda\leq X^{1/36}}}\bigl(\Vol(B(n,s_1,s_2,
  \lambda,X))s_1^{-6}s_2^{-6}dn\, d^\times s\,d^\times \lambda\, dk=
\frac{1}{C_{G_0}}\int_{h\in
  G_0}\Vol(\pnRh)dh\\&&-O\Bigl(\int_{\substack{na(s_1,s_2)\lambda k\in\FF/\FF'}}\lambda^{30}s_1^{-6}s_2^{-6}dn\,d^\times s\,d^\times \lambda+
\int_{\substack{na(s_1,s_2)\lambda k\in\FF\\\lambda<c_1}}\lambda^{30}s_1^{-6}s_2^{-6}dn\,d^\times s\,d^\times \lambda\Bigl),
\end{eqnarray*}
since $\Vol(B(n,s_1,s_2,\lambda,X))=O(\lambda^{30})$.
The error term in the above equation is computed to be $O(X^{2/3})$.
Hence it follows from Equation (\ref{eqtcomp}), and the fact that $\Vol(\pnRh)$ is independent of $h$, that
\begin{equation}\label{eqfinvol}
N(\pnv;X)=\frac{1}{3}\Vol(\pnRR\,\!)+o(X^{5/6}).
\end{equation}
Therefore, to prove Theorem \ref{thsec2main}, it remains only to compute the
volume $\Vol(\pnRR\,\!)$.

\subsection{Computing the volume}
In this section, we compute the volumes of $\pnRR$. To this end, let
$R^\pm:=\Lambda\cdot L^\pm$. Then the sets $R^\pm$ consist of one
element in $V_\R^\pm$ having invariants $(I,J)$ for each
$(I,J)\in\R\times\R$ such that $\pm\Delta(I,J)\in\R_{>0}$. Let
$R^\pm(X)$ denote the set of points in $R^\pm$ having height bounded
by $X$. Consider the following action of the group $\PGL_3(\R)$ on $V_\R$ given by
$$
\gamma\cdot f(x,y,z):=\frac1{\det(\gamma)}f((x,y,z)\cdot\gamma),
$$
for $\gamma\in\PGL_3(\R)$ and $f\in V_\R$. Let $\FF_{\PGL_3}$ denote
the image in $\PGL_3(\R)$ of $\FF$. Then $\FF_{\PGL_3}$ is a
fundamental domain for the action of $\PGL_3(\Z)$ on $\PGL_3(\R)$. We
have $\pnRR=\FF_{\PGL_3}\cdot R^\pm(X)$.

Let $\omega$ be a differential which generates the rank $1$ module of
top-degree differentials of $\PGL_3$ over $\Z$. Then $\omega$ is well
defined up to sign. To compute the volume of the multiset
$\FF_{\PGL_3}\cdot R^\pm(X)$, we have the following proposition:
\begin{proposition}\label{bqjac}
For any measurable function $\phi$
  on $V_\R$, we have
\begin{equation}\label{Jac}
\frac{4}{9}\int_{R^\pm}
\int_{\PGL_3(\R)}\phi(g\cdot p_{I,J})\,\omega(g)\,dI dJ
=
\int_{\PGL_3(\R)\cdot R^\pm}\phi(v)dv
=
3\int_{V_\R^\pm}\phi(v)dv.
\end{equation}
where $p_{I,J}\in R^\pm$ is the point having invariants equal to $I$
and $J$ and we regard $\PGL_3(\R)\cdot R^{\pm}$ as a multiset.
\end{proposition}
The above proposition may be verified by a direct Jacobian computation. We
will also give a more noncomputational proof in \S\ref{s32}.

Since the volume of $\FF_{\PGL_3}$ is equal to $3\zeta(2)\zeta(3)$ (see \cite{Langlands}), we have:
\begin{equation}
  \int_{\pnRR}\!dv=\int_{\FF_{\PGL_3}\cdot R^\pm(X)}\!dv=
  \frac{4}{9}\int_{R^\pm(X)}\int_{\FF_{\PGL_3}}\omega(g)\,dI\,dJ=\frac{4\zeta(2)\zeta(3)}{3}
\int_{R^\pm(X)}dI\,dJ.
\end{equation}
The quantity $\int_{R^+(X)}dI\,dJ$ is equal to
\begin{equation}\label{volrplus}
\int_{I=0}^{X^{1/3}}\int_{J=-2I^{3/2}}^{2I^{3/2}}dJ\,dI=\int^{X^{1/3}}_{I=0}4I^{3/
2}dI 
=\frac{8}{5}X^{5/6},
\end{equation}
while $\int_{R^-(X)}dI\,dJ$ is equal to
\begin{equation}\label{volrminus}
\int_{I=-X^{1/3}}^{X^{1/3}}\int_{J=-2X^{1/2}}^{2X^{1/2}}dJ\,dI-\int_{\rrpx}dI\,dJ=8X
^{5/6}-\frac{8}{5}X^{5/6}=\frac{32}{5}X^{5/6}.
\end{equation}
We conclude that 
\begin{equation}
\begin{array}{ccc}
\Vol(\pRR\,\!)&=&\displaystyle{\frac{32\zeta(2)\zeta(3)}{15}X^{5/6}},\\[.15in]
\Vol(\nRR\,\!)&=&\displaystyle{\frac{128\zeta(2)\zeta(3)}{15}X^{5/6}},
\end{array}
\end{equation}
which along with (\ref{eqfinvol}) yields Theorem~\ref{thsec2main}.

\subsection{Congruence conditions}\label{congse}
In this subsection, we prove a version of Theorem \ref{thsec2main} where we
count integral ternary cubic forms satisfying any finite set of
congruence conditions.

For any set $S$ in $V_\Z$ that is definable by congruence conditions,
we denote by $\mu_p(S)$ the $p$-adic density of the $p$-adic closure
of $S$ in $V_{\Z_p}$, where we normalize the additive measure $\mu_p$
on $V_{\Z_p}$ so that $\mu_p(V_{\Z_p})=1$.  We then have the following
theorem whose proof is identical to that of \cite[Theorem 2.11]{BS}.
\begin{theorem}\label{cong2}
Suppose $S$ is a subset of $\pnv$ defined by finitely many
congruence conditions. Then we have 
\begin{equation}\label{ramanujan}
N(S\cap\pnv;X)
  = N(\pnv;X)
  \prod_{p} \mu_p(S)+o(X^{5/6}),
\end{equation}
where $\mu_p(S)$ denotes the $p$-adic density of $S$ in $V_\Z$, and
where the implied constant in $o(X^{5/6})$ depends only on $S$.
\end{theorem}

We will also have occasion to use the following weighted version of
Theorem \ref{cong2}; the proof is identical to that of \cite[Theorem
  2.12]{BS}.

\begin{theorem}\label{cong3}
  Let $p_1\ldots,p_k$ be distinct prime numbers. For $j=1,\ldots,k$,
  let $\phi_{p_j}:V_\Z\to\R$ be a $\SL_3(\Z)$-invariant function on
  $V_\Z$ such that $\phi_{p_j}(f)$ depends only on the congruence
  class of $f$ modulo some power $p_j^{a_j}$ of $p_j$.  Let
  $N_\phi(V_\Z^{\pm};X)$ denote the number of strongly irreducible
  $\SL_3(\Z)$-orbits in $V_\Z^{\pm}$ having height bounded by $X$,
  where each orbit $\SL_3(\Z)\cdot f$ is counted with weight
  $\phi(f):=\prod_{j=1}^k\phi_{p_j}(f)$. Then we have
\begin{equation}
N_\phi(V_\Z^{\pm};X)
  = N(V_\Z^{\pm};X)
  \prod_{j=1}^k \int_{f\in V_{\Z_{p_j}}}\tilde{\phi}_{p_j}(f)\,df+o(X^{5/6}),
\end{equation}
where $\tilde{\phi}_{p_j}$ is the natural extension of ${\phi}_{p_j}$
to $V_{\Z_{p_j}}$ by continuity, $df$ denotes the additive
measure on $V_{\Z_{p_j}}$ normalized so that $\int_{f\in
  V_{\Z_{p_j}}}df=1$, and where the implied constant in the error term
depends only on the local weight functions ${\phi}_{p_j}$.
\end{theorem}

\subsection[The number of reducible points in the main bodies of the
  fundamental domains is negligible]{The number of reducible points
    and points with large stabilizers in
  the main bodies of the fundamental domains is negligible}

In this section, we prove Lemma \ref{lemfewred}, i.e., that the number
of $\SL_3(\Z)$-orbits of reducible elements in $V_\Z$ of bounded height is
negligible.  Via a similar argument, we also show that the
number of $\SL_3(\Z)$-orbits of strongly irreducible elements in
$V_\Z$ having nontrivial stabilizer in $\PGL_3(\Q)$ and bounded
height is negligible. We use the technique in the proof of \cite[Lemma~14]{dodpf}.

\vspace{.1in}
\noindent {\bf Proof of Lemma \ref{lemfewred}:}
  Suppose $f$ is an integral ternary cubic form. If $f$ has a rational
  flex in $\P^2$, then for any prime $p$, the
  reduction $\bar{f}$ of $f$ modulo $p$ has a point of inflection in
  $\P^2(\F_p)$. Now let $p$ be a 
prime that is congruent to $1$ (mod
  $3$), and let $a$, $b$, $c$ be elements in $\F_p^\times$ that are in
  different cube classes (i.e., none of $a/b$, $b/c$, $c/a$ are cubes
  in $\F_p^\times$). Then one easily checks that the ternary cubic
  form $f_{a,b,c}(x,y,z)=ax^3+by^3+cz^3\in V_{\F_p}$ has no point of
  inflection in $\F_p$. Hence none of the forms in the set
  $S_p=\{\gamma\cdot f_{a,b,c}:\gamma\in\GL_3(\F_p)\}$ contain points
  of inflection in $\F_p$. It is clear that $\#S_p\gg p^9$, where the
  implied constant is independent of $p$. Thus, if $s_p$ denotes the
  $p$-adic density of the set of elements in $V_\Z$
  whose reduction modulo $p$ is contained in $S_p$, then $s_p\gg
  p^9/p^{10}=1/p$, where the implied constant is independent of $p$.
  Therefore, for any $Y>0$, we have
\begin{equation}\label{eqsec2lempr}
\int_{na(s_1,s_2)\lambda k\in\FF'}\#\{V_\Z^{\red}\cap
B(n,s_1,s_2, \lambda,X)\}s_1^{-6}s_2^{-6}dn\, d^\times t\,d^\times
\lambda dk\ll X^{5/6}\prod_{\substack{p\equiv 1\!\!\!\!\pmod{3}\\p\leq Y}}(1-s_p).
\end{equation}
Since $s_p\gg 1/p$, it follows that $\prod_{p\equiv
  1\!\!\!\pmod{3}}(1-s_p)$ diverges, and hence the left hand side of
(\ref{eqsec2lempr}) is $o(X^{5/6})$ as required. $\Box$ \vspace{2 ex}

We may use the same method to bound the number of $\SL_3(\Z)$-orbits
on strongly irreducible integral ternary cubic forms having bounded
height and nontrivial stabilizer in $\PGL_3(\Q)$, i.e., integral
ternary cubic forms of bounded height whose associated cubic curves
have no rational flex in $\P^2$, but whose Jacobians possess a
nontrivial $3$-torsion point defined over $\Q$.
\begin{lemma}\label{lemtemp2}
Let $V_\Z^\bigstab$ denote the set of elements in $V_\Z$ whose stabilizer in $\PGL_3(\Q)$ is nontrivial.
Then $N(V_\Z^\bigstab;X)=o(X^{5/6})$.
\end{lemma}
\begin{proof}
  By Equation (\ref{eqavgimp}) and Lemma \ref{lemnoirred}, it suffices 
  to prove the estimate
  \begin{equation}\label{eqlemred2}
    \int_{na(s_1,s_2)\lambda k\in\FF'}\#\{V_\Z^{\bigstab}\cap B(n,s_1,s_2,
    \lambda,X)\}s_1^{-6}s_2^{-6}dn\, d^\times t\,d^\times \lambda dk=o(X^{5/6}).
  \end{equation}
  The Jacobian of a form $f\in V_\Z$ may be embedded in $\P^2$ as a
  Weierstrass elliptic curve $\Jac(f)$ via the equation
  \begin{equation*}
    y^2z=x^3-\frac{I}{3}xz^2-\frac{J}{27}z^3,
  \end{equation*}
  and under this embedding, the $3$-torsion points of the Jacobian of
  $f$ are precisely the flex points in $\P^2$ of the curve $\Jac(f)$.
  Thus an integral ternary cubic form $f$ is contained in
  $V_\Z^{\bigstab}$ if and only if the curve $\Jac(f)$ contains at least
  two rational flex points in $\P^2$. The proof of the estimate in
  (\ref{eqlemred2}) now proceeds very similarly to that of Lemma
  \ref{lemfewred}. The only difference is that we now consider, for
  each $p\equiv 7$ (mod 12), the form $f_b(x,y,z)=x^3+bz^3-y^2z\in V_{\F_p}$, where
  $b$ is a nonresidue in $\F_p$. We see that $\Jac(f)$ is then
  precisely the curve
  defined by the equation $f=0$, and it has exactly one inflection
  point in $\P^2(\F_p)$, namely the point $[0:1:0]$.
\end{proof}

\subsection[The average number of strongly irreducible integral
  ternary cubic forms with given invariants]{The average number of strongly irreducible integral
  ternary cubic forms with given invariants (Proofs of Theorems
  \ref{eligible} and \ref{bqaverage})}

We first prove Theorem \ref{eligible} by describing the set of
{\it eligible} pairs $(I,J)\in\frac{1}{16}\Z\times\frac{1}{32}\Z$, i.e., those pairs that occur as
invariants of integral ternary cubic forms. We begin by showing that a
pair $(I,J)$ is eligible if and only if it occurs as the invariants
of a Weierstrass elliptic curve over $\Z$:
\begin{proposition}\label{thinttcjac}
  A pair $(I,J)$ is eligible if and only if it occurs as the
  invariants of some Weierstrass cubic over $\Z$, where a Weierstrass
  cubic over $\Z$ is an element in $V_\Z$ of the form
$$y^2z+a_1xyz+a_3yz^2-x^3-a_2x^2z-a_4xz^2-a_6z^3.$$
\end{proposition}
\begin{proof}
  This proposition is easily deduced from the results in \cite{ATR}. To
  any integral ternary cubic form $f\in V_\Z$, one may associate a
  Weierstrass ternary cubic form $f^*$ {over $\Z$} which defines
  the Jacobian curve (see \cite[Equation 1.5]{ATR}). 
  The invariants $c_4(f)$ and $c_6(f)$ of $f$ are then defined to be
  equal to the classical invariants $c_4(f^*)$ and $c_6(f^*)$ of the
  corresponding Weierstrass cubic (see \cite{Sil} for a definition of $c_4(f^*)$ and
  $c_6(f^*)$). Using \cite[Equation 1.7]{ATR}, we easily check that
  our invariants $I(f)$ and $J(f)$ are equal to the invariants $c_4(f)/16$ and
  $c_6(f)/32$, respectively, for any ternary cubic form $f$. We conclude
  that $I(f)=I(f^*)$ and $J(f)=J(f^*)$, as desired. 
\end{proof}

Next, we have a result of Kraus (see \cite[Proposition 2]{Kr}) which
describes those pairs $(c_4,c_6)$ that can occur for a
Weierstrass cubic over $\Z$:
\begin{proposition}\label{thkr}
  Let $c_4$ and $c_6$ be integers. In order for there to exist a
  Weierstrass cubic over $\Z$ having nonzero discriminant and invariants
  $c_4$ and $c_6$, it is necessary and sufficient that
  \begin{itemize}
  \item[{\rm (a)}]$(c_4^3-c_6^2)/1728$ is a nonzero integer$;$
  \item[{\rm (b)}]$c_6\not\equiv \pm9\pmod{27};$
  \item[{\rm (c)}]either $c_6\equiv -1\pmod 4$, or $c_4\equiv 0
    \pmod{16}$ and $c_6\equiv 0,8\pmod{32}.$
  \end{itemize}
\end{proposition}

It can be checked that the set of pairs $(I,J)$ that satisfy the
congruence conditions of Theorem~\ref{eligible} is the same as the set
of pairs $(c_4/16,c_6/32)$ for which the congruence conditions of
Proposition~\ref{thkr} are satisfied for $(c_4,c_6)$. Thus
Theorem~\ref{eligible} follows from Propositions~\ref{thinttcjac} and
\ref{thkr}, and the fact that $I(f)=c_4(f)/16$ and $J(f)=c_6(f)/32$
for Weierstrass cubics $f$ having integral coefficients.

The next lemma follows immediately from Theorem \ref{eligible}.
\begin{lemma}\label{elij}
  The set of all eligible $(I,J)$ is a
  union of $144$ distinct translates of $36\Z\times27\Z$ in $\frac1{16}\Z\times\frac1{32}\Z$.
\end{lemma}
\begin{proposition}\label{propijcount}
  Let $N^+_{I,J}(X)$ and $N^-_{I,J}(X)$ denote the number of eligible
  pairs $(I,J)\in\frac1{16}\Z\times\frac1{32}\Z$ satisfying $H(I,J)<X$
  that have positive discriminant and negative discriminant,
  respectively. Then
  \begin{itemize}
  \item[{\rm (a)}]$\displaystyle{N^+_{I,J}(X)=\frac{32}{135}X^{5/6}+O(X^{1/2})};$
  \item[{\rm (a)}]$\displaystyle{N^-_{I,J}(X)=\frac{128}{135}X^{5/6}+O(X^{1/2})}.$
  \end{itemize}
\end{proposition}
\begin{proof}
  Let $R_{I,J}^+(X)$ (resp.\ $R_{I,J}^-(X)$) denote the set of points $(i,j)\in
  \R\times\R$ satisfying $H(i,j)<X$ and $4i^3-j^2>0$ (resp.\
  $4i^3-j^2<0$). The
  sizes of the projections of $R_{I,J}^{\pm}(X)$ onto
  smaller-dimensional coordinate hyperplanes are all bounded by
  $O(X^{1/2})$.  Using Proposition~\ref{lemdav} and Lemma~\ref{elij}
  then gives
$$N^\pm_{I,J}(X)=\frac{144}{36\cdot
  27}\Vol(R_{I,J}^\pm(X))+O(X^{1/2}).$$
The volumes of the sets $R_{I,J}^+(X)$ and $R_{I,J}^-(X)$ have been computed in
(\ref{volrplus}) and (\ref{volrminus}) to be equal to $8/5$
and $32/5$, respectively. The proposition follows.
\end{proof}

Theorem~\ref{bqcount} combined with Proposition \ref{propijcount}
now yields Theorem \ref{bqaverage}.

\subsection{Uniformity estimates and a squarefree sieve}

For our applications, we require a general version of Theorem
\ref{cong3}, namely one that counts weighted ternary cubic forms where
the weight functions are defined by appropriate infinite sets of
congruence conditions. 
A function $\phi:V_{\Z}\to[0,1]\in\R$ is said to be {\it defined
  by congruence conditions} if, for all primes $p$, there exist
functions $\phi_p:V_{\Z_p}\to[0,1]$ satisfying the following
conditions:
\begin{itemize}
\item[(1)] For all $f\in V_\Z$, the product $\prod_p\phi_p(f)$ converges to $\phi(f)$.
\item[(2)] For each prime $p$, the function $\phi_p$ is 
locally constant outside some closed set $S_p \subset V_{\Z_p}$ of measure zero.
\end{itemize}
We say that such a function $\phi$ is {\it acceptable} if for
sufficiently large primes $p$, we have $\phi_{p}(f)=1$ whenever
$p^2\nmid\Delta(f)$.

Our purpose in this section is to prove the following generalization of Theorem \ref{cong3}.
\begin{theorem}\label{thsqfreetc}
  Let $\phi:V_\Z\to[0,1]$ be an acceptable function that is defined by
  congruence conditions via the local functions $\phi_{p}:V_{\Z_p}\to[0,1]$. Then, with
  notation as in Theorem~$\ref{cong3}$, we have:
\begin{equation}
N_\phi(V_\Z^\pm;X)
  = N(V_\Z^\pm;X)
  \prod_{p} \int_{f\in V_{\Z_{p}}}\phi_{p}(f)\,df+o(X^{5/6}).
\end{equation}
\end{theorem}

To prove Theorem~\ref{thsqfreetc}, we need the following tail estimate.
\begin{proposition}\label{propuniftc}
  Let $\W_p(V)$ denote the set of ternary cubic forms $f$ such that $p^2\mid\Delta(f)$. Then, for any fixed $\epsilon>0$, we have
\begin{equation}\label{tailestimate}
N(\cup_{p>M}\W_p;X)=O_\epsilon(X^{5/6}/(M\log M)+X^{3/4})+O(\epsilon X^{5/6}).
\end{equation}
\end{proposition}
\begin{proof}
  Let $\W_p^{(1)}$ denote the set of ternary cubic forms such that
  $p^2\mid \Delta(f)$ for ``mod $p$ reasons'', i.e., $p^2\mid
  \Delta(g)$ for every $g\equiv f\pmod{p}$. For any $\epsilon>0$, let
  $\FF_{\PGL_3}^{(\epsilon)}\subset\FF_{\PGL_3}$ denote the subset of
  elements $na(s_1,s_2)k\in\FF_{\PGL_3}$ such that $s_1$ and $s_2$ are
  bounded above by an appropriate constant to ensure that
  $\Vol(\FF_{\PGL_3}^{(\epsilon)})=(1-\epsilon)\Vol(\FF_{\PGL_3})$. Then
  $\FF_{\PGL_3}^{(\epsilon)}\cdot R^\pm(X)$ is a bounded domain in
  $V_\R$ that expands homogeneously with
  $X$. By~\cite[Theorem~3.3]{geosieve}, we have
\begin{equation}\label{e1}
\#\{\FF_{\PGL_3}^{(\epsilon)}\cdot R^\pm(X)\bigcap (\cup_{p>M}\W^{(1)}_p)\}=O(X^{5/6}/(M\log M)+X^{9/12}).
\end{equation}
Furthermore, the results of \S\ref{s21} and \S\ref{s22} imply that
\begin{equation}\label{e2}
\#\{(\FF_{\PGL_3}\backslash\FF_{\PGL_3}^{(\epsilon)})\cdot R^\pm(X)\bigcap V_\Z^\irr\}=O(\epsilon X^{5/6}).
\end{equation}
Combining the two estimates (\ref{e1}) and (\ref{e2}) 
yields \eqref{tailestimate} with $\W_p$
replaced with $\W^{(1)}_p$.

Next, suppose $f$ belongs to $\W^{(2)}_p:=\W_p\backslash\W^{(1)}_p$. Let
$\bar{f}$ denote the reduction of $f$ modulo $p$. Then the curve
$C\subset\P^2_{\F_p}$ defined by $\bar{f}(x,y,z)=0$ contains a single
nodal singularity. This singularity must be $\F_p$-rational and we can
move it to $[0:0:1]$ using an element of $\SL_3(\F_p)$. In that case,
the $z^3$-, $xz^2$-, and $yz^2$-coefficients of $\bar{f}(x,y,z)$ are
zero. Evaluating the discriminant of an element $f$ that reduces mod
$p$ to such an $\bar{f}$, we see that $\Delta(f)\equiv
cG(f)\pmod{p^2}$, where $c$ is the coefficient of $z^3$ and $G(f)$ is
an irreducible polynomial in the coefficients of $f$. As $f\in
\W_p^{(2)}$, we see that $G(f)\not\equiv0\pmod{p}$. Therefore, since
$p^2\mid \Delta(f)$, we obtain that $p^2\mid c$, the coefficient of
$z^3$. Now the element $g$ defined by
\begin{equation}\label{eqmatgamma}
    \left(\begin{smallmatrix}
      1&&\\&1&\\&&p^{-1}
    \end{smallmatrix}\right)\cdot pf
\end{equation}
has the same discriminant as $f$ and is in $\W_p^{(1)}$, because its $x^3$-, $x^2y$-, $xy^2$-, and $y^3$-coefficients are zero modulo~$p$. We therefore
obtain a discriminant preserving map $\phi$ from $\SL_3(\Z)$-orbits on
$\W_p^{(2)}$ to $\SL_3(\Z)$-orbits on $\W_p^{(1)}$. The following lemma
states that this map is at most $3$ to $1$:
\begin{lemma}
Given an $\SL_3(\Z)$-orbit on $\W_p^{(1)}$, there are at most three
$\SL_3(\Z)$-orbits on $\W_p^{(2)}$ that map to it under $\phi$.
\end{lemma}
\begin{proof}
If the reduction of $f\in\W^{(2)}_p$ modulo $p$ has a nodal
singularity at $[0:0:1]\in\P^2(\F_p)$, then the form in $W_p^{(1)}$ given by 
\eqref{eqmatgamma}, when reduced modulo $p$, has $z$ as a factor. 
Moreover, for $g\in W_p^{(1)}$, the form
\begin{equation}\label{eqmatgammainverse}
    \left(\begin{smallmatrix}
      1&&\\&1&\\&&p
    \end{smallmatrix}\right)\cdot p^{-1}g
\end{equation}
can be integral only if the $x^3$-, $x^2y$-, $xy^2$-, and
$y^3$-coefficients of $g$ are zero modulo $p$.  Therefore, the preimages
under $\phi$ of the $\SL_3(\Z)$-orbit of $g\in\W_p^{(1)}$ are
associated to 
linear factors of the reduction of $g$ modulo $p$. The reduction of
$g$ modulo $p$ has at most $3$ linear factors, unless $g\equiv
0\pmod{p}$. However, if $g\equiv 0\pmod{p}$, then it is easy to see that
\eqref{eqmatgammainverse} belongs to $\W_p^{(1)}$. Thus, the map
$\phi:\SL_3(\Z)\backslash\W_p^{(2)}\to \SL_3(\Z)\backslash\W_p^{(1)}$
is at most $3$ to~$1$, and the lemma follows.
\end{proof}

Therefore, we obtain
\begin{equation}
N(\cup_{p>M}\W_p^{(2)};X)\leq
3N(\cup_{p>M}\W_p^{(1)};X)=O_\epsilon(X^{5/6}/(M\log M)+X^{3/4})+O(\epsilon X^{5/6}).
\end{equation}
This concludes the proof of the proposition.
\end{proof}

Theorem \ref{thsqfreetc} now follows from Proposition
\ref{propuniftc} just as 
\cite[Theorem 2.21]{BS} followed from \cite[Theorem~2.13]{BS}.

\section{The average number of elements in the $3$-Selmer groups of elliptic curves}

Recall that any isomorphism class of elliptic curve $E$ over $\Q$ has
a unique representative of the form
\begin{equation}\label{eqellip}
E(A,B): y^2=x^3+Ax+B,
\end{equation}
where $A,B\in\Z$ and for all primes $p$, we have $p^4\nmid A$ if
$p^6\mid B$.  For any elliptic curve $E(A,B)$ over $\Q$ written in the
form (\ref{eqellip}), we define the quantities $I(E)$ and $J(E)$ by
\begin{equation}I(E)=-3A,\end{equation} 
\begin{equation}J(E)=-27B.\end{equation} We denote the 
elliptic curve over $\Q$ 
having invariants $I$ and $J$ by $E^{I,J}$, and define its
{\it height} $H'(E^{I,J})$ by
$$H'(E^{I,J}):=\max\{|I^3|,J^2/4\}.$$ 
We use the height $H'$ instead of $H$ on elliptic curves to agree with
the height on integral ternary cubic forms defined in
\eqref{eqhttc}. Note that since $H$ and $H'$ agree up to a constant
factor, they induce the same ordering on the set of elliptic curves
over $\Q$.

In this section, we prove Theorem \ref{ellipcong} by computing the
average size of the $3$-Selmer group of elliptic curves over $\Q$,
whose coefficients satisfy finitely many congruence conditions, when
these curves are ordered by their heights. We also prove a theorem
where we bound the average size of the $3$-Selmer group of elliptic
curves in more general families. To define these families, we need the
following notation.  For each prime~$p$, let $\Sigma_p$ be a closed
subset of $\Z_p^2\backslash\{\Delta\neq 0\}$. We associate the family
$F_\Sigma$ of elliptic curves to $(\Sigma_p)_p$, where $E^{I,J}\in
F_\Sigma$ if $(I,J)\in\Sigma_p$ for all $p$. Such a family is said to
be {\it defined by congruence conditions}. We can also impose
``congruence conditions at infinity'' by insisting that $E^{I,J}\in
F_\Sigma$ if and only if $(I,J)\in\Sigma_\infty$, where
$\Sigma_\infty$ is equal to $\{(I,J)\in\R^2:\Delta(I,J)>0\}$,
$\{(I,J)\in\R^2:\Delta(I,J)<0\}$, or $\{(I,J)\in\R^2:\Delta(I,J)\neq
0\}$.

If $F$ is a family of elliptic curves defined by congruence
conditions, then let $\Inv(F)$ denote the set $\{(I(E),J(E)):E\in
F\}$.  For a prime $p$, let $\Inv_p(F)$ denote the $p$-adic closure of
$\Inv(F)$ in $\Z_p^2$. We define $\Inv_\infty(F)$ to be
$\{(I,J)\in\R^2:\Delta(I,J)>0\}$,
$\{(I,J)\in\R^2:\Delta(I,J)<0\}$, or $\{(I,J)\in\R^2:\Delta(I,J)\neq
0\}$ in accordance with whether $F$ contains curves only of positive
discriminant, negative discriminant, or both.  
A family $F$ of elliptic curves is then said to be
{\it large} if, for all but finitely many primes $p$, the set
$\Inv_p(F)$ contains all pairs $(I,J)\in\Z_p\times\Z_p$ such that
$p^2\nmid\Delta(I,J)$.

In this section, we prove the following theorem:
\begin{theorem}\label{ellipgen}
  When elliptic curves $E$ in any large family are ordered by
  height, the average size of the $3$-Selmer group $S_3(E)$ is $4$.
\end{theorem}

\subsection{Ternary cubic forms and elements in the $3$-Selmer groups of elliptic curves}\label{s31}

For a field $K$, we may define a twisted action of the group
$\GL_3(K)$ on the space $V_K$ of
ternary cubic forms having coefficients in $K$ via 
\begin{equation}
\gamma\cdot
f(x,y,z):=\det(\gamma)^{-1}f((x,y,z)\cdot\gamma),
\end{equation} which induces an
action of $\PGL_3(K)$ on $V_K$.   We say that a
ternary cubic form $f\in V_K$ is {\it $K$-soluble} if the equation
$f(x,y,z)=0$ has a nontrivial solution over $K$.
We then have the following result, which follows from
\cite[Theorem~2.5 and Remark~2.7]{Fisher2} (see also \cite[\S4.2]{BhHo}):
\begin{proposition}\label{propparamK}
  Let $K$ be a field having characteristic not equal to $2$ or $3$,
  and $E:y^2=x^3-\frac{I}{3}x-\frac{J}{27}$ be an elliptic curve over
  $K$.  Then there exists a natural injection
  $$\mathcal T_E\,:\,E(K)/3E(K) \,\rightarrow\, 
  \{\PGL_3(K)\textrm{-orbits of ternary cubic forms
    over}\;K\},$$ whose image consists exactly of the $K$-soluble
  ternary cubic forms having invariants equal to $I$ and~$J$.  
Under this correspondence, the identity element of $E(K)/3E(K)$
maps to the $\PGL_3(K)$-orbit of ternary cubic forms having a
$K$-rational point of inflection. 

  Furthermore, the stabilizer in $\PGL_3(K)$ of any ternary cubic form
  having invariants equal to~$I$ and~$J$ is isomorphic to $E(K)[3]$.
\end{proposition}

A rational ternary cubic form $f\in V_\Q$ is said to be {\it locally
  soluble} if $f$ is $\R$-soluble and $\Q_p$-soluble for all primes
$p$. We then have the following proposition (see
\cite[Remark~2.8]{Fisher2}):

\begin{proposition}\label{propselparq}
  Let $E/\Q$ be an elliptic curve. Then the elements in the $3$-Selmer
  group of $E$ are in bijective correspondence with
  $\PGL_3(\Q)$-orbits on the set of locally soluble ternary cubic
  forms in $V_\Q$ having invariants equal to $I(E)$ and $J(E)$.

  Furthermore, the set of all ternary cubic forms in $V_\Q$ having
  invariants equal to $I(E)$ and $J(E)$ that are not strongly
  irreducible lie in a single $\PGL_3(\Q)$-orbit, and this orbit
  corresponds to the identity element in the $3$-Selmer group of $E$.
\end{proposition}

By a result of Cremona, Fisher, and Stoll \cite[Theorem 1.1]{CFS}, any
rational ternary cubic form $f\in V_\Q$ having integral invariants $I$
and $J$ is $\SL_3(\Q)$-equivalent to an integral ternary cubic form
$g\in V_\Z$ having invariants $I$ and $J$. In particular, it follows
that such an $f$ is $\PGL_3(\Q)$-equivalent to either $g$ or
$-g$. Since $g$ and $-g$ have the same invariants, we obtain the
following proposition:

\begin{proposition}\label{propselparz}
  Let $E/\Q$ be an elliptic curve. Then the elements in the $3$-Selmer
  group of $E$ are in bijective correspondence with
  $\PGL_3(\Q)$-equivalence classes\footnote{We refer to the set of all
    $g\in V_\Z$ such that $g$ is $\PGL_3(\Q)$-equivalent to a fixed
    integral ternary cubic form $f$ as the {\it $\PGL_3(\Q)$-equivalence
    class} of $f$ in $V_\Z$.} of locally soluble integral ternary cubic
  forms in $V_\Z$ having invariants equal to $I(E)$ and $J(E)$.

  Furthermore, the set of all ternary cubic forms in $V_\Z$ having
  invariants equal to $I(E)$ and $J(E)$ that are not strongly
  irreducible lie in a single $\PGL_3(\Q)$-equivalence class, and this
  equivalence class corresponds to the identity element in the
  $3$-Selmer group of~$E$.
\end{proposition}

We may use Proposition \ref{propparamK} to
prove Lemma \ref{lemstabsize}, i.e., that the order of the stabilizer
of a real ternary cubic form of nonzero discriminant is always equal
to 3.
\\[.1in]
{\bf Proof of Lemma \ref{lemstabsize}:} Let
$f\in V_\R$ be a ternary cubic form having nonzero discriminant. By
Proposition~\ref{propparamK}, if $f$ has invariants $I$ and $J$, then the size
of the stabilizer of $f$ in $\PGL_3(\R)$ is equal to the number of
real $3$-torsion points on the elliptic curve $E^{I,J}/\R$ having invariants
$I$ and $J$. Now the $3$-torsion points of a plane Weierstrass 
elliptic curve are its
flex points, and it is known 
that
any plane cubic curve over $\R$ having nonzero discriminant
has exactly $3$ flex points defined over $\R$ (see, e.g.,
\cite[Chapter~13]{Gib}). Furthermore, if
$\gamma\in\GL_3^+(\R)$ stabilizes $f$, then $I(\gamma\cdot f)=(\det
\gamma)^4I(f)$ and $J(\gamma\cdot f)=(\det \gamma)^6J(f)$, and so $\det
\gamma=1$. This completes the proof of Lemma~\ref{lemstabsize}.
$\Box$\vspace{.1 in}

\subsection{A change-of-measure formula}\label{s32}

We begin with the following proposition, which is an extension of the
change-of-measure formula in Proposition~\ref{bqjac} so that it holds
also over $\Z_p$, with $4/9$ replaced by $|\J|$ for some rational
constant $\J$.
\begin{proposition}\label{propjacone}
  Let $K$ be $\R$ or $\Z_p$ for some prime $p$, let $|\cdot|$ denote the
  usual absolute value on $K$, and let $s:K^2\to V_K$ be a continuous
  section. Then there exists a rational nonzero constant $\J$,
  independent of $K$ and $s$, such that for any measurable function
  $\phi$ on $V_K$, we have
\begin{equation}
  \begin{array}{rcl}
    \displaystyle\int_{\PGL_3(K)\cdot
      s(K^2)}\!\!\!\phi(f)df&=&  |\J|\displaystyle\int_{(I,J)\in K^2}\displaystyle\int_{g\in \PGL_3(K)}
    \phi(g\cdot s(I,J))\omega(g) dIdJ;\\[0.3in]
    \displaystyle\int_{V_K}\phi(f)df&=&|\J|\displaystyle\int_{\substack{(I,J)\in K^2\\ \Delta(I,J)\neq 0\}}}
\Bigl(\displaystyle\sum_{f\in\textstyle{\frac{V_K(I,J)}{\PGL_3(K)}}}\frac{1}{\#\Stab(f)}\int_{g\in \PGL_3(K)}\phi(g\cdot f)\omega(g)\Bigr)dIdJ,
  \end{array}
\end{equation}
where $\Stab(f)$ denotes the stabilizer of $f$ in $\PGL_3(K)$ and
$\frac{V_K(I,J)}{\PGL_3(K)}$ denotes a set of representatives for the
action of $\PGL_3(K)$ on the set of elements in $V_K$ having
invariants $I$ and $J$.
\end{proposition}
Proposition~\ref{propjacone}
follows immediately from the proofs of \cite[Propositions~3.11 and 3.12]{BS} 
(see \cite[Remark 3.14]{BS}).

In the rest of the section, we compute the value of $\J$. To do so, we
use the following proposition whose proof is identical to that 
of \cite[Proposition 3.13]{BS}:
\begin{proposition}\label{propjacvol1}
  Let $p$ be a fixed prime number. Let $S\subset
  V_{\Z_p}$ be a set defined by congruence conditions modulo $p$, and let $\bar{
S}\subset
  V_{\F_p}$ denote the reduction of $S$ modulo $p$. Assume that $S=\pi^{-1}(\pi(
S))$, where $\pi$ is given by taking invariants. Then
\begin{equation}\label{eqcompjfirst}
|\J|_p=
\frac{\#\PGL_3(\F_p)\cdot\Bigl(\displaystyle\sum_{f\in\PGL_3(\F_p)\backslash\bar{S}}\displaystyle\frac{1}{\#\Aut_{\F_p}(f)}\Bigr)}
{p^{\dim V}\cdot\Vol(\PGL_3(\Z_p))\cdot\Bigl(\displaystyle\int_{(I,J)\in\pi(S)}
\displaystyle\sum_{f\in\textstyle{\frac{V_{\Z_p}(I,J)}{\PGL_3(\Z_p)}}}\frac{1}{\#
\Aut_{\Z_p}(f)}dIdJ\Bigr)}.
\end{equation}
\end{proposition}
Note that the numerator in the right hand side of \eqref{eqcompjfirst}
is equal to $\#\bar{S}$ by the orbit-stabilizer formula.  

The next two
lemmas allow us to evaluate the numerator and the denomenator of
\eqref{eqcompjfirst} for the set $S\subset V_{\Z_p}$ consisting of
elements whose discriminants are prime to $p$.

\begin{lemma}\label{lemmaprop1}
  Let $p$ be a fixed prime. Then the number of elements in $V_{\F_p}$
  that have nonzero discriminant is equal to $(p^2-p)\cdot\#\PGL_3(\F_p)$.
\end{lemma}
\begin{proof}
  Ternary cubic forms over $\F_p$ having nonzero discriminant
  correspond to isomorphism classes of 
triples $(C,L,B)$, where $C$ is a genus $1$ curve over
  $\F_p$, $L$ is a degree $3$ line bundle on $C$, and $B$ is a basis
  for the space of sections of $L$; here, two such pairs $(C,L,B)$ and
  $(C',L',B')$ are called isomorphic if there exists an isomorphism
  $\phi:C\to C'$ such that $L=\phi^\ast(L')$ and $B=\phi^\ast(B')$. 
The number of such isomorphism classes of pairs $(C,L)$ over
  $\F_p$ is exactly $p$, since there is exactly one pair for each
  $j$-invariant. Once the pair $(C,L)$ is fixed, there are
  $\#\GL_3(\F_p)$ different possible bases for the space of
  sections. Since $\#\GL_3(\F_p)=(p-1)\#\PGL_3(\F_p)$, we obtain the
  lemma.
\end{proof}

\begin{lemma}\label{lemmaprop}
  Let $p$ be a fixed prime and let $(I,J)\in \frac1{16}\Z_p\times\frac1{32}\Z_p$ be an element in
  the image of $\pi$ such that $p^2\nmid\Delta(I,J)$. Then
$$
\displaystyle\sum_{f\in\textstyle\frac{V_{\Z_p}(I,J)}{\PGL_3(\Z_p)}}\frac1{\#\Aut_{\Z_p}(f)}=\left\{\begin{array}{ll}1 \mbox{ for $p\neq 3$;}\\[0.03in]3 \mbox{ for $p=3$.}\end{array}\right.
$$
\end{lemma}
\begin{proof}
Since $p^2\nmid \Delta(I,J)$, we have
$\Aut_{\Z_p}(f)=\Aut_{\Q_p}(f)=E^{I,J}[3](\Q_p)$ for $f\in
V_{\Z_p}(I,J)$. Furthermore, we have
$$
\#\frac{V_{\Z_p}(I,J)}{\PGL_3(\Z_p)}=\#(E(\Q_p)/3E(\Q_p)).
$$
The lemma therefore follows from Lemma~\ref{lembk} in \S\ref{s34}.
\end{proof}

We may now compute the value of $|\J|_p$ using Proposition
\ref{propjacvol1}. For each prime $p$, we pick $S$ to be the set of
elements $f\in V_{\Z_p}$ such that $p\nmid \Delta(f)$. Since
$\Vol(\PGL_3(\Z_p))=\#\PGL_3(\F_p)/p^8$, Equation \eqref{eqcompjfirst}
in conjunction with Lemmas \ref{lemmaprop1} and \ref{lemmaprop} implies that
$$
|\J|_p=|3|_p\frac{p-1}{p\Vol(\pi(S))}.
$$
We may now use Theorem \ref{eligible} to compute the volume of
$\pi(S)$, yielding
$$
\Vol(\pi(S))=\left\{\begin{array}{cl} \frac{p-1}{p} &\mbox{ if $p\geq 5$;}\\[0.08in]\frac{2}{81} &\mbox{ if $p=3$;}\\[0.08in]2 &\mbox{ if $p=2$.}\end{array}\right.
$$
We conclude that $\J=4/9$, as desired.

\subsection{Computations of $p$-adic densities in terms of local masses}
Proposition \ref{propselparz} asserts that the nonidentity elements in
the $3$-Selmer group of $E^{I,J}$ are in bijection with
$\PGL_3(\Q)$-orbits on the set of strongly irreducible locally soluble
integral ternary cubic forms having invariants $I$ and $J$.  In
Section 2, we computed the asymptotic number of
$\SL_3(\Z)$-equivalence classes of strongly irreducible integral
ternary cubic forms having bounded height.  In order to use this to
compute the number of $\PGL_3(\Q)$-equivalence classes of strongly
irreducible locally soluble integral ternary cubic forms having
bounded height, we must count each locally soluble orbit
$\PGL_3(\Z)\cdot f$ weighted by $1/n(f)$, where $n(f)$ is the number
of $\PGL_3(\Z)$-orbits in the $\PGL_3(\Q)$-equivalence class of $f$ in
$V_\Z$. Since we have shown that all but a negligible number of
integral ternary cubic forms have trivial stabilizer in $\PGL_3(\Q)$,
we may instead count each locally soluble orbit $\PGL_3(\Z)\cdot f$
weighted by $1/m(f)$, where
$$m(f):=\sum_{f'\in B(f)}\frac{\#\Aut_\Q(f')}{\#\Aut_\Z(f')}=\sum_{f'\in B(f)}\frac{\#\Aut_\Q(f)}{\#\Aut_\Z(f')};$$
here $B(f)$ is a set of representatives for the action of $\PGL_3(\Z)$
on the $\PGL_3(\Q)$-equivalence class of $f$ in $V_\Z$, while
$\Aut_\Q(f)$ and $\Aut_\Z(f)$ are the stabilizers of $f$ in
$\PGL_3(\Q)$ and $\PGL_3(\Z)$, respectively.

For a prime $p$ and an ternary cubic form $f\in V_{\Z_p}$, define $m_p(f)$ by 
$$m_p(f):=\sum_{f'\in B_p(f)}\frac{\#\Aut_{\Q_p}(f')}{\#\Aut_{\Z_p}(f')}=\sum_{f'\in B(f)}\frac{\#\Aut_{\Q_p}(f)}{\#\Aut_{\Z_p}(f')},$$
where $B_p(f)$ denotes a set of representatives for the action of
$\PGL_3(\Z_p)$ on the $\PGL_3(\Q_p)$-equivalence class of $f$ in
$V_{\Z_p}$, while $\Aut_{\Q_p}(f)$ (resp.\ $\Aut_{\Z_p}(f)$) denotes
the stabilizers of $f$ in $\PGL_3(\Q_p)$ (resp.\ $\PGL_3(\Z_p)$).
Then we have the following proposition which explains the advantage of
using the weights $m(f)$ rather than~$n(f)$, and whose proof is
identical to that of \cite[Proposition~3.6]{BS}.

\begin{proposition}\label{locglob}
Suppose $f\in V_\Z$ has nonzero discriminant. Then $m(f)=\prod_pm_p(f)$.
\end{proposition}

Suppose now that $F$ is a large family of elliptic
curves. Recall that we denoted the set $\{(I(E),J(E):E\in F\}$ by
$\Inv(F)$ and the $p$-adic closure of $\Inv(F)$ in $\Z_p^2$ by
$\Inv_p(F)$.
Let $S(F)$ denote the set of all locally soluble integral ternary cubic forms having invariants $I$ and $J$ such that
$(I,J)\in \Inv(F)$, and let $S_p(F)$ denote the
$p$-adic closure of $S(F)$ in $V_{\Z_p}$.  We now determine the
$p$-adic density of $S_p(F)$, where each element $f\in
S_p(F)$ is weighted by $1/m_p(f)$,
in terms of a {\it local $(p$-adic$)$ mass} $M_p(V,F)$ involving all
isomorphism classes of soluble $3$-coverings of elliptic curves over
$\Q_p$:

\begin{proposition}\label{denel}
We have
$$\int_{S_p(F)}\frac{1}{m_p(f)}df=|4/9|_p\Vol(\PGL_3(\Z_p))M_p(V,F),$$
where
\begin{equation}\label{mpvfdef}
M_p(V,F)=\displaystyle{\int_{(I,J)\in \Inv_p(F)}\frac{\#(E^{I,J}(\Q_p)/3E^{I,J}(\Q_p))}{\#E^{I,J}(\Q_p)[3]}}dIdJ.
  \end{equation}
\end{proposition}
The proof of Proposition~\ref{denel} is identical to that of
\cite[Proposition~3.9]{BS}.

\subsection{The average size of the $3$-Selmer groups of elliptic curves in a large family (Proof of Theorem \ref{ellipgen})}\label{s34}

In analogy with $M_p(V,F)$, we define the local mass $M_p(F)$ by
\begin{equation}\label{mpufdef}
M_p(F)=\int_{(I,J)\in \Inv_p(F)}dIdJ.
\end{equation}
We also define the following analogues at infinity of $M_p(F)$ and
$M_p(V,F)$, respectively.
\begin{equation}\label{eqminf}
  \begin{array}{rcl}
    M_\infty(F;X)&:=&\displaystyle\int_{\substack{(I,J)\in\Inv_\infty(F)\\H(I,J)
<X}}dIdJ,\\[.3in]
    M_\infty(V,F;X)&:=&\displaystyle\int_{\substack{(I,J)\in\Inv_\infty(F)\\H(I,
J)<X}}\displaystyle\frac{\#(E^{I,J}(\R)/3E^{I,J}(\R))}{\# E^{I,J}(\R)[3]}dIdJ.
  \end{array}
\end{equation}
We then have the following result counting the number of elliptic
curves in a large family, which is \cite[Theorem~3.17]{BS}:
\begin{theorem}\label{thnumelip}
Let $F$ be a large family of elliptic curves and let $N(F;X)$
denote the number of elliptic curves in $F$ that have
height bounded by $X$. Then
\begin{equation}\label{eqnumelip}
N(F;X)=M_\infty(F;X)\prod_pM_p(F)+o(X^{5/6}).
\end{equation}
\end{theorem}

We say that an element $f\in V_\Z$ is {\it bad at $p$} if either $f$
is not $\Q_p$-soluble or $m_p(f)\neq 1$. To deduce Theorem~\ref{ellipgen} using Theorem~\ref{thsqfreetc}, we need the following result.

\begin{proposition}\label{badatp}
Let $f$ be an integral ternary cubic form such that either
$f$ is insoluble at $p$ or $m_p(f)\neq 1$. Then $p^2\mid \Delta(f)$.
\end{proposition}
\begin{proof}
  Suppose that $f$ is an integral ternary cubic form such that the
  curve defined by $f(x,y,z)=0$ has no $\Q_p$-points. We claim that
  $f$ is geometrically reducible over $\F_p$ (i.e., $f\pmod{p}$
  factors into a product of lower degree forms defined over
  $\overline{\F}_p$). This is because if $f$ were geometrically
  irreducible over $\F_p$, then the Lang-Weil estimates \cite{LW}
  would imply
  that the curve $f(x,y,z)=0$ has a smooth point in $\P^2(\F_p)$. By
  Hensel's lemma, this smooth point lifts to a point in
  $\P^2(\Q_p)$. Therefore, $f$ is geometrically reducible over $\F_p$,
  implying that $p^2\mid \Delta(f)$.

  If $f\in V_\Z$ satisfies $m_p(f)\neq 1$, then there exists an element
  $\gamma\in\PGL_3(\Q_p)$ such that $\gamma\cdot f\in V_{\Z_p}$.  By
  an appropriate change-of-basis in $\PGL_3(\Z_p)$, we may assume that
  $\gamma$ is of the form $\gamma=\left(
    \begin{smallmatrix}
      p^{-a}&&\\&\!\!\!1&\\&&\,p^{b}
    \end{smallmatrix}\right)$, 
  where $a$ and $b$ are nonnegative and at least one of $a$ and $b$ is
  nonzero.  If $b>a$, then clearly $z$ is a factor of the reduction
  of $f$ modulo~$p$.  Now assume that $a\geq b$, so that $a>0$. In
  this case, consider the form $f_1=\left(
    \begin{smallmatrix}
      p^{-1}&&\\&\!\!1&\\&&\,1
    \end{smallmatrix}\right)\cdot pf$
  which is an element of $V_{\Z_p}$ having the same invariants as
  $f$. The reduction of $f_1$ modulo~$p$, being a multiple of $x$, is
  a reducible ternary cubic form. We conclude that $p^2\mid
  \Delta(f_1)=\Delta(f)$.
\end{proof}

Theorem \ref{ellipgen} will be deduced from the following result:
\begin{theorem}\label{th38}
Let $F$ be a large family of elliptic curves. Then
\begin{equation}\label{eqthlast}
\displaystyle{\lim_{X\to\infty}\displaystyle\frac{\displaystyle
    \sum_{\substack{E\in
        F\\H'(E)<X}}(\#S_3(E)-1)}{\displaystyle\sum_{\substack{E\in
        F\\H'(E)<X}}1}}=\Vol(\PGL_3(\Z)\backslash\PGL_3(\R))\frac{M_\infty(V,F;X)}{M_\infty(F;X)}\prod_p\Bigl[\Vol(\PGL_3(\Z_p))\frac{M_p(V,F)}{M_p(U_1,F)}\Bigr].
\end{equation}
\end{theorem}
\begin{proof}
  The numerator of the right hand side of \eqref{eqthlast} is equal to
  the number of $\PGL_3(\Z)$-orbits on $S(F)$ having height bounded by
  $X$, where each orbit $\PGL_3(\Z)\cdot f$ is counted with weight
  $1/m(f)$.
  Therefore, by Theorem~\ref{thsqfreetc}, Proposition \ref{locglob},
  and Proposition \ref{badatp}, we obtain
\begin{equation}\label{eqn3s}
\begin{array}{rcl}
  \displaystyle
 \sum_{\substack{E\in
        F\\H'(E)<X}}(\#S_3(E)-1) &\!\!\!=\!\!\!&N(V_\Z\cap S_\infty(F);X)\displaystyle\prod_p\int_{S_p(F)}\frac1{m_p(f)}df+o(X^{5/6})\\
  &\!\!\!=\!\!\!&\displaystyle\frac49\Vol(\PGL_3(\Z)\backslash\PGL_3(\R))M_\infty(V,F;X)\displaystyle\prod_p\bigl[\Bigl|\frac49\Bigr|_p\Vol(\PGL_3(\Z_p)) M_p(V,F)\bigr]+o(X^{5/6})\\[.2in]&\!\!\!=\!\!\!&
\Vol(\PGL_3(\Z)\backslash\PGL_3(\R))M_\infty(V,F;X)\displaystyle\prod_p\bigl[\Vol(\PGL_3(\Z_p)) M_p(V,F)\bigr]+o(X^{5/6}),
\end{array}
\end{equation}
where the second equality follows from Proposition \ref{denel}.
Taking the ratio of (\ref{eqn3s}) and (\ref{eqnumelip}), we obtain the
theorem.
\end{proof}

To evaluate the right hand side of (\ref{eqthlast}) we need the
following fact whose proof is identical to that of \cite[Lemma 3.1]{BK}:
\begin{lemma}\label{lembk}
Let $E$ be an elliptic curve over $\Q_p$. We have
$$\#(E(\Q_p)/3E(\Q_p))=
\left\{\begin{array}{rcl}
\#E[3](\Q_p) & {\mbox{\em if }} p\neq 3;\\[.1in]
3\cdot\#E[3](\Q_p)& {\mbox{\em if }} p=3.
\end{array}\right.$$
\end{lemma}

\begin{proof}
  A well-known result of Lutz (see, e.g., \cite[Chapter~7,
  Proposition~6.3]{Sil} for a proof) asserts that there exists a subgroup
  $M\subset E(\Q_p)$ of finite index that is isomorphic to $\Z_p$. Let
  $G$ denote the finite group $E(\Q_p)/M$.
 Then by applying the snake lemma to the following diagram
$$\xymatrix{
0\ar[d]\ar[r]&M\ar[d]^{[3]}\ar[r]&E(\Q_p)\ar[d]^{[3]}\ar[r]&G\ar[d]^{[3]}\ar[r]&0\ar[d]\\
0\ar[r]&M\ar[r]&E(\Q_p)\ar[r]&G\ar[r]&0}$$
we obtain the exact sequence
$$0\to M[3]\to E(\Q_p)[3]\to G[3]\to M/3M\to E(\Q_p)/3E(\Q_p)\to G/3G\to 0.$$
Since $G$ is a finite group and $M$ is isomorphic to $\Z_p$, Lemma \ref{lembk} follows.
\end{proof}

By Lemma \ref{lembk} and the definitions of $M_p(V,F)$ and $M_p(U_1,F)$, we have
\begin{equation}
\frac{M_p(V,F)}{M_p(F)}\;=\;\displaystyle\frac{\displaystyle\displaystyle\int_{(I,J)\in \Inv_p(F)}\frac{\#(E^{I,J}(\Q_p)/3E^{I,J}(\Q_p))}{\#E^{I,J}(\Q_p)[3]}dIdJ}
{\displaystyle\displaystyle\int_{(I,J)\in \Inv_p(F)}dIdJ}\;=\;\left\{\begin{array}{ll}
1
&\quad\mbox{if $p\neq 3$;}\\[.1in]
3
&\quad\mbox{if $p=3$.}
\end{array}\right.
\end{equation}
Furthermore, we know that $M_\infty(V,F;X)/M_\infty(F;X)=1/3$. Therefore, Theorem \ref{th38} yields
\begin{eqnarray*}
\displaystyle{\lim_{X\to\infty}\displaystyle\frac{\displaystyle
    \sum_{\substack{E\in
        F\\H'(E)<X}}(\#S_3(E)-1)}{\displaystyle\sum_{\substack{E\in
        F\\H'(E)<X}}1}}=\Vol(\PGL_3(\Z)\backslash\PGL_3(\R))\prod_p\Vol(\PGL_3(\Z_p))
\end{eqnarray*}
which is then equal to $3\zeta(2)\zeta(3)\prod_p\bigl((1-p^{-2})(1-p^{-3})\bigr)=3$, the Tamagawa number of $\PGL_3(\Q)$.
We have proven Theorem \ref{ellipgen}.

\section{A positive proportion of elliptic curves have rank $0$}

We have shown in the previous section that the average rank of all
elliptic curves, when ordered by height, is less than $
1\frac16
$.  This immediately implies that a large proportion (indeed, at
least $62.5\%$) of all elliptic curves must have rank 0 {\it or} 1.

In order to deduce analogous positive proportion statements for the
individual ranks~0 and~1, we may attempt to make use of information
regarding the distribution of the {\it parity} of the ranks---or of
the $3$-Selmer ranks---of these curves.  Indeed, if we knew that even
and odd 3-Selmer ranks occur equally often in a large family of
elliptic curves, then this would imply by Theorem~\ref{ellipgen} that
a positive proportion of curves in that family have rank~0, and
(assuming finiteness of the Tate--Shafarevich group) a positive
proportion have rank~1.

In Section 4.1, we use a recent result of Dokchitser--Dokchitser~\cite{DD} 
(see also Nekov\'a\v{r}~\cite{Nekovar}) 
to construct a large,
positive proportion family $F$ of elliptic curves in which the
parities of the 3-Selmer ranks of the curves in $F$ are equally
distributed between even and odd, thus unconditionally yielding a
positive proportion of elliptic curves having rank~0.

We may also combine our counting techniques with the recent work of
Skinner--Urban~\cite{SU}, in order to deduce that a positive
proportion of all elliptic curves, when ordered by height, have {\it
  analytic rank} \,0; i.e., a positive proportion of all elliptic
curves have nonvanishing $L$-function $L(E,s)$ at $s=1$. Since these
analytic rank 0 curves form a subset of the rank 0 curves of
Section 4.1, it follows that a positive proportion of all elliptic
curves satisfy the Birch and Swinnerton-Dyer conjecture.  This is
discussed in Section 4.2.

\subsection{Elliptic curves having algebraic rank 0}

Recall that the conjecture of Birch and Swinnerton-Dyer implies, in
particular, that the evenness or oddness of the rank of an elliptic
curve $E$ is determined by whether its {\it root number}---that is,
the sign of the functional equation of the $L$-function $L(E,s)$ of
$E$---is $+1$ or $-1$, respectively.  It is widely believed that the
root numbers $+1$ and $-1$ occur equally often among all elliptic
curves when ordered by height.  Indeed, we expect the same to be true
in any large family as well.

In this subsection, we prove:
\begin{theorem}\label{pralrank}
  Suppose $F$ is a large family of elliptic curves such that
  exactly $50\%$ of the curves in $F$, when ordered by height, have
  root number $+1$.  Then at least $25\%$ of the curves in $F$, when
  ordered by height, have rank $0$.  Furthermore, if we assume that
  all the elliptic curves in $F$ have finite Tate-Shafarevich groups,
  then at least $5/12> 41.6\%$ of the curves in $F$ have rank $1$.
\end{theorem}
We will construct an explicit positive proportion family $F$
satisfying the hypotheses of Theorem~\ref{pralrank}; this will then
imply Theorem~\ref{algrank0}.  (Of course, it is expected that the
family $F$ of all curves satisfies the root number hypothesis of the
theorem; however, this remains unproved.)

Our proof of Theorem~\ref{pralrank} is based on Theorem~\ref{ellipgen}
in conjunction with a recent remarkable result of Dokchitser and
Dokchitser~\cite{DD} which asserts (as predicted by the Birch and
Swinnerton-Dyer conjecture) that
the parity of the $p$-Selmer rank of an elliptic curve $E$ (for any prime $p$) is determined by the root number of $E$:

\begin{theorem}[Dokchitser--Dokchitser]\label{thDD}
Let $E$ be an elliptic curve over $\Q$ and let $p$ be any prime. 
Let $s_p(E)$ and $t_p(E)$ denote the rank of the $p$-Selmer group of $E$ and the rank of $E(\Q)[p]$, respectively.
Then the quantity $r_p(E):=s_p(E)-t_p(E)$ is even if and only if the root
number of $E$ is~$+1$.
\end{theorem}

We now prove Theorem~\ref{pralrank}.

\vspace{.15in}\noindent {\bf Proof of Theorem~\ref{pralrank}:} First
note that Lemma \ref{lemtemp2} implies that the number of elliptic
curves over $\Q$ that have a nontrivial rational $3$-torsion point is
negligible. Thus for a density of $100\%$ of elliptic curves $E$, we have
$r_p(E)=s_p(E)$.

Now, by Theorem~\ref{ellipgen}, the average size of the $3$-Selmer group of
curves in $F$ is at most $4$. On the other hand, by Theorem~\ref{thDD}
we know that that exactly $50\%$ of the curves in $F$ have odd
$3$-Selmer rank and thus have at least $3$ elements in the $3$-Selmer
group. Hence the average size of the $3$-Selmer groups among the $50\%$
of elliptic curves in $F$ having even $3$-Selmer rank is at most $5$.
Now if the 3-Selmer group of an elliptic curve has even rank, then it
must have size 1, 9, or more than 9.  For the average of such sizes to
be 5, at least half must be equal to 1.  Thus among these $50\%$ of
curves in $F$ having even 3-Selmer rank, at least half have trivial
$3$-Selmer group, and therefore have rank~$0$.

Next, suppose that every odd rank curve in $F$ has a finite
Tate-Shafarevich group. A well-known result of Cassels states that if
$E/\Q$ is an elliptic curve such that $\SH(E)$ is finite, then
$|\SH(E)|$ is a square.  Now if the 3-Selmer group of an elliptic curve
has odd rank, then it must have size 3, 27, or more than 27.  For the
average of such sizes to be at most $7$, at least 5/6 of them must equal 3.
Thus among these $50\%$ of curves in $F$ with odd 3-Selmer rank, at
least 5/6 of them have $3$-Selmer group of size 3.  Since $\SH$ is
always a square, we conclude that $\SH[3]$ for all these elliptic
curves is trivial and so they each have rank 1.
$\Box$\vspace{.15in}

We now construct an explicit positive proportion large family $F$
of elliptic curves for which exactly $50\%$ of the curves have root
number equal to $1$.  By Theorem~\ref{pralrank}, this will then imply
Theorems~\ref{algrank0} and \ref{algrank1}.

First, recall that the root number $\omega(E)$ of an elliptic curve
$E$ over $\Q$ may be expressed in terms of a product over all primes
of local root numbers $\omega_p(E)$ of $E$, namely,
$\omega(E)=-\prod_p\omega_p(E)$.  The local root number $\omega_p(E)$
is easy to compute when $E$ has good or multiplicative reduction at
$p$. In fact, it is known (see, e.g., \cite{Rh}) that $\omega_p(E)=1$
whenever $E$ has good or non-split multiplicative reduction at $p$,
and $\omega_p(E)=-1$ when $E$ has split multiplicative reduction at
$p$.

Suppose an elliptic curve $E/\Q$ has multiplicative reduction at a
prime $p\geq 3$. Then it is easily checked that $E$ has split
reduction precisely when $\bigl(\frac{-2J}{p}\bigr)=1$. It is also
clear that if $E_{-1}$ denotes the twist of $E$ over $\Q[i]$, then
$J(E_{-1})=-J(E)$. Hence, given an odd prime $p$ for which $E$ has
multiplicative reduction at $p$, we have
$\omega_p(E)=\omega_p(E_{-1})$ if and only if $p\equiv 1\pmod{4}$.

Let $F$ denote the set of all elliptic curves $E$ over $\Q$ satisfying
the following conditions:
\begin{itemize}
\item The curve $E$ and its twist $E_{-1}$ both have additive reduction
  at $2$, and furthermore the $j$-invariant of both curves $E$ and $E_{-1}$ are
  $2$-adic units.
\item $E$ has square-free discriminant away from $2$.
\item $\Delta'(E)\equiv 1\pmod4$, where
  $\Delta'(E):=|\Delta(E)/2^{v_2(\Delta(E))}|$ is the positive odd
  part of the discriminant of $E$.
\end{itemize}
The set $F$ is a large family. Moreover, if $E\in F$, then the
twist $E_{-1}$ of $E$ by $-1$ is also clearly in $F$, since the odd
part of the discriminant of an elliptic curve is preserved under such
a twist. Since $\Delta'(E)$ is squarefree, the third condition implies
that the number of distinct primes congruent to 3 (mod~4) that divide
the discriminant of $E$ is even. Now for a prime factor $p$ of the
discriminant, we have already observed that
$\omega_p(E)=-\omega_p(E_{-1})$ if and only if $p\equiv 3\pmod{4}$.
Furthermore, the first condition implies that
$\omega_2(E)=-\omega_2(E_{-1})$ (see \cite[Lemma 12]{SW}); therefore,
$\omega(E)=-\omega(E_{-1})$ for all $E\in F$.  Since the height of an
elliptic curve also remains the same under twisting by $-1$, it
follows that a density of exactly $50\%$ of elliptic curves in $F$,
when ordered by height, have root number~$+1$, as desired.

We have proven Theorems~\ref{algrank0} and \ref{algrank1}.  

\subsection{Elliptic curves having analytic rank  0}

We may similarly prove that a positive proportion of all elliptic
curves have analytic rank 0, by combining our counting arguments with
the recent beautiful work of Skinner--Urban~\cite{SU}.  Their work
implies, in particular,
that if $E/\Q$ is an elliptic curve satisfying certain mild conditions
and having trivial $3$-Selmer group (and therefore rank $0$), then the
$L$-function of $E$ does not vanish at the point $1$!  The following
theorem is a consequence of~\cite[Theorem~2]{SU}:
\begin{theorem}[Skinner--Urban]\label{thSU}

Let $E/\Q$ be an elliptic curve such that:
\begin{itemize}
\item[{\rm (a)}]The $3$-Selmer group of $E$ is trivial;
\item[{\rm (b)}]$E$ has good ordinary reduction at $3$;
\item[{\rm (c)}]The action of $G_\Q$ on $E[3]$ is irreducible.
\item[{\rm (d)}]There exists a prime $p\neq 3$ such that $p\|{\rm
    Cond}(E)$ and $\bar{\rho}(E,3)$ is ramified at $p$,
\end{itemize}
where $\bar{\rho}(E,3):{\rm Gal}(\bar{\Q}/\Q)\to\GL_2(\F_3)$ denotes
the usual Galois representation obtained from the action of ${\rm
  Gal}(\bar{\Q}/\Q)$ on the $3$-torsion points of $E$. Then
$L(E,1)\neq 0$.
\end{theorem}
We may use Theorem~\ref{pralrank} in conjunction with Skinner and
Urban's Theorem to prove:

\begin{theorem}\label{pranrank0}
  Suppose $F$ is a large family of elliptic curves having good
  ordinary reduction at $3$ such that $5\|\Disc(E)$ for every curve
  $E\in F$. Further assume that exactly $50\%$ of the curves in $F$,
  when ordered by height, have root number $+1$. Then at least $25\%$
  of elliptic curves in $F$ have analytic rank $0$.
\end{theorem}

\begin{proof}
It is easy to see (e.g., by Hilbert irreducibility) that a density of 
$100\%$ of elliptic curves~$E$, when ordered by height,
  have the property that the action of $G_\Q$ on $E[3]$ is
  irreducible. (In fact, 
it has been shown by Duke~\cite[Theorem 1]{Du} 
  that $100\%$ of all elliptic curves~$E$, when ordered by height,
  have the property that the action of $G_\Q$ on $E[p]$ is
  irreducible for all primes $p$.)
 As $F$ is a large family, it contains a positive proportion of all
  elliptic curves, and so $100\%$ of the curves in $F$
  satisfy condition (c) of Theorem~\ref{thSU}.  As $5\|\Disc(E)$ for
  $E\in F$, we see that $E$ has multiplicative reduction at $5$ which
  implies that $5\|{\rm Cond}(E)$. Furthermore, since $3\nmid v_5({\rm
    Cond}(E))$, \cite[Proposition 2.12]{DDT} implies that condition (d)
  of Theorem~\ref{thSU} is satisfied by $E$. The proof of
  Theorem~\ref{pralrank} now implies that at least $25\%$ of the curves in $F$
  satisfy all four conditions of Theorem~\ref{thSU}, and so 
Theorem~\ref{pranrank0} follows.
\end{proof}

As in \S4.1, we may construct an explicit union $F$ of positive proportion
large families of elliptic curves satisfying the hypotheses of
Theorem~\ref{pranrank0}.  Indeed, let $F$ denote the family of all
elliptic curves $E$ satisfying the following conditions:
\begin{itemize}
\item The curve $E$ and its twist $E_{-1}$ both have additive reduction
  at $2$, and furthermore the $j$-invariant of both curves $E$ and $E_{-1}$ are
  $2$-adic units.
\item $E$ has square-free discriminant away from $2$, and $5\|\Disc(E)$.
\item $E$ has good ordinary reduction at $3$.
\item $\Delta'(E)\equiv 1\pmod4$, where
  $\Delta'(E):=|\Delta(E)/2^{v_2(\Delta(E))}|$ is the positive odd
  part of the discriminant of $E$.
\end{itemize}
Then, just as in \S4.1, we see that $50\%$ of the curves in $F$ have
root number $+1$.  Thus, by Theorem~\ref{pranrank0}, a positive
proportion of these and thus all elliptic curves, when ordered by
height, have both algebraic and analytic rank 0; we have proven
Theorem~\ref{anrank0} and Corollary~\ref{bsdcor}.

\subsection*{Acknowledgments}

We are very grateful to John Cremona, Johan de Jong, Tom Fisher, Wei
Ho, Bjorn Poonen, Shrenik Shah,
Christopher Skinner, Michael Stoll, Damiano Testa, Eric Urban, and Jerry Wang 
for helpful conversations.  The first author was partially supported
by NSF Grant~DMS-1001828.

\end{document}